\documentclass[11pt,reqno]{amsart}
\usepackage{amssymb,amsmath,amsthm,amscd,amsfonts }
\usepackage{enumerate}
\usepackage{color}

\newtheorem{thm}{Theorem}[section]
\newtheorem{thmA}{Theorem}

\newtheorem{cor}[thm]{Corollary}
\newtheorem{lem}[thm]{Lemma}

\newtheorem{prop}[thm]{Proposition}
\newtheorem{quest}[thm]{Question}

\stepcounter{equation}

\theoremstyle{definition}
\newtheorem{defn}[thm]{Definition}
\newtheorem{rmk}[thm]{Remark}
\newtheorem{exa}[thm]{Example}

\newcommand{\mb}[1]{\mathbb{#1}}

\newcommand{\im}{\operatorname{im}}
\renewcommand{\phi}{\varphi}
\newcommand{\cl}{\operatorname{cl}}
\newcommand{\Ball}{{\mathrm{B}}}
\renewcommand{\d}{\operatorname{dist}}

\newcommand{\conn}{\operatorname{Con}^\omega}

\newcommand{\cone}{\operatorname{Cone}}

\begin{document}

\date{}
\title{Asymptotic cones and boundaries of CAT(0) groups}
\author{Curtis Kent \and Russell Ricks}\thanks{The first author was supported by Simons Foundation Collaboration Grant 587001. The second author was partially supported by NSF RTG 1045119.}

%\subjclass[2010]{Primary  20F65; Secondary 20F69  , }

\begin{abstract}
 It is well known that the Euclidean cone over the Tits boundary of a proper cocompact CAT(0) space isometrically embeds into every asymptotic cone of the space.  We explore the relationships between the asymptotic cones of a CAT(0) space and its boundary under both the standard visual (i.e. cone) topology and the Tits metric.  We show that the set of asymptotic cones of a proper cocompact CAT(0) space admits canonical connecting maps under which the direct limit is isometric to the Euclidean cone on the Tits boundary.  The resulting projection from any asymptotic cone to the Tits boundary is determined by the visual topology; on the other hand, the visual topology can be recovered from the connecting maps between asymptotic cones.  We also demonstrate how maps between asymptotic cones induce maps between Tits boundaries.
\end{abstract}

\maketitle
%\subjclass[2010]{Primary 57N05; Secondary 55P10, 54E45,20F34}

%\keywords{Peano continuum, fundamental group, planar}

\setcounter{section}{0}

\section{Introduction }

Asymptotic cones and compactifications are two very different
%, apparently antithetical,
 approaches to studying the coarse geometry of groups and spaces.  In general, asymptotic cones of a metric space $X$ will not be locally compact and not preserve the local geometry of $X$ while compactifications often can be endowed with metrics topologically compatible with the local geometry of $X$.
However, asymptotic cones do reflect well the coarse Euclidean aspects of a space; and for CAT(0) spaces, the asymptotic cones are complete CAT(0) spaces.

%The geometry of a finitely generated group is only well-defined up to quasi-isometry.  Thus when studying groups via their actions on metric spaces, it is convenient to consider constructions that interpret the large scale geometry.  Asymptotic cones of finitely generated groups and boundaries of hyperbolic and CAT(0) spaces are two such constructions.
%A reoccurring question in geometric group theory is to what extend does the combinatorial and algorithmic properties of a group determine the geometric structure of a space on which the group acts by isometries.  Since the geometry of a finitely generated group is only well-defined up to quasi-isometry, one can only expect to obtain information concerning the coarse geometric structure of a space by considering finitely generated groups acting on it.  Thus, it is convenient to consider constructions that interpret the large scale geometry of metric spaces.  Asymptotic cones and boundaries of hyperbolic and CAT(0) spaces  are two such constructions.  Asymptotic cones are one way to view a metric space from infinitely far away.

Intuitively, an asymptotic cone of a metric space $(X, \d)$ is the limit of the metric spaces $(X, \d/d_n)$ where $\d/d_n$ is the metric on $X$ scaled by $1/d_n$.  Asymptotic cones have the desirable property that quasi-isometric spaces have bi-Lipschitz asymptotic cones.  However, an ultrafilter is required to guarantee the existence of the limit.     As a consequence, they also have the drawback that a metric space can have uncountably many distinct asymptotic cones depending on the choice of ultrafilter and scaling sequence used in the construction.

There are many connections between the topological structure of the asymptotic cones of a finitely generated group and its combinatorial and algorithmic properties.  For example, a group has polynomial growth if and only if it is virtually nilpotent, if and only if every asymptotic cone is locally compact \cite{Gromov1,Pansu}.  Also, a finitely generated group is finitely presented and has polynomial Dehn function if all of its asymptotic cones are simply connected \cite{gr2}.
%The classes of hyperbolic and relative hyperbolic groups can be characterized in terms of their asymptotic cones \cite{gr2, DS}.

Another approach to studying the coarse geometry of groups and spaces is via boundaries.  The visual boundary is one natural compactification for proper CAT(0) and hyperbolic spaces.  The visual boundary is the set of large-scale directions and is an isometry invariant of the space.  In fact, for hyperbolic groups the visual boundary is a quasi-isometry invariant.  However, Croke and Kleiner proved that a group can act geometrically on two CAT(0) spaces with non-homeomorphic visual boundaries \cite{CrokeKleiner00}.  Thus in the CAT(0) setting the visual boundary is not a quasi-isometry invariant.

The boundary of a CAT(0) space can also be endowed with a metric called the Tits metric which reflects the Euclidean structure of the CAT(0) space.  With the Tits metric, the set of large-scale directions is no longer a compactification but in some sense better encodes the coarse Euclidean geometry of the space.

 %Another approach to studying the large scale geometric of CAT(0) and hyperbolic groups is to consider the boundaries.   The boundary of a can be endowed with to natural topologies, the visual and the Tits topologies.  The visual topology gives a compactification of the group while the Tits topology is generally not give a compactification.

%Another approach to studying the large scale geometric is to consider compactifications of metric spaces.  One canonical approach to finding a compactification for hyperbolic and CAT(0) spaces is to attach the visual boundary.  The visual boundary is the set of geodesic rays in a geodesic metric space.  The visual boundary is a quasi-isometry invariant for hyperbolic groups.  However, Croke and Kleiner \cite{CrokeKleiner00} illustrated a group $G$ acting geometrically on two CAT(0) spaces with non-homeomorphic visual boundaries.

We will illustrate how the visual boundary, the Tits boundary, and the asymptotic cones of proper CAT(0) spaces relate.  It is well known that the Euclidean cone over the Tits boundary admits a canonical isometric embedding into every asymptotic cone.  We will show that the set of asymptotic cones of a proper CAT(0) space determine the Tits boundary of the space.

\begin{thmA}[Theorem \ref{direct limit}]
Let $X$ be a proper cocompact CAT(0) space.   The direct limit of asymptotic cones of $X$ induced by the geodesic retraction on $X$ is isometric to the Euclidean cone on the Tits boundary of $X$.
Moreover, the resulting projection maps onto the direct limit are determined by the visual topology on the boundary.
\end{thmA}

\noindent Thus the set of asymptotic cones of $X$ for a fixed ultrafilter determines the Tits boundary of $X$.  The connecting maps also give rise to an inverse limit, which is related to the countable ultraproduct of the  Tits boundary.

\begin{thmA}[Theorem \ref{inverse limit}]
Let $X$ be a proper cocompact CAT(0) space.   The inverse limit of asymptotic cones of $X$ induced by the geodesic retraction on $X$ has an inverse limit metric and with this metric there exists a canonical isometric embedding of  the ultraproduct of the Euclidean cone on the Tits boundary of $X$ into the inverse limit.
\end{thmA}

\noindent We leave it as an open question whether or not this embedding is surjective.

In Section \ref{visual section}, we demonstrate how the asymptotic cones determine the visual topology on the boundary.

\begin{thmA}[Theorem \ref{visual recovery}]
The visual topology on the boundary of a proper cocompact CAT(0) space is determined by the geodesic retraction maps between asymptotic cones.
\end{thmA}

In Section \ref{quasi-isometries}, the direct limit is used to define continuous maps between Tits boundaries of quasi-isometric CAT(0) spaces, which when restricted to Morse geodesics gives a bijection.  This gives an alternate proof that a CAT(0) group has a cut-point in some asymptotic cone if and only if it has cut-points in every asymptotic cone if and only if it has a periodic rank one element.

%The authors would also like to thank Kim Ruane and the 2013 AMS Mathematical Research Communities where this project was started.

\section{Visual boundary}

A CAT(0) space is a uniquely geodesic metric space such that every geodesic triangle $\triangle (x,y,z)$ is thinner than the corresponding comparison triangle $\overline \triangle (x,y,z)$ in Euclidean $\mathbb R^2$.  This generalizes the property of nonpositive curvature from Riemannian manifolds to the metric setting. We refer the reader to \cite{BridHae99} for a more complete account.

The visual boundary of a complete CAT(0) space can be considered as either the set of equivalence classes of geodesic rays (equivalent if they are asymptotic) or the set of based geodesic rays.  Here we will use the latter. %since it has a canonical embedding into every asymptotic cone (see Definition \ref{embedding}).

For a CAT(0) space $(X,\d)$ and $x,y\in X$, we will use $[x,y]$ to denote the unique unparameterized geodesic from $x$ to $y$, and $\Ball^{\d}_\epsilon(x)$ or $\Ball_\epsilon(x)$ to denote the open metric ball of radius $\epsilon$ about $x$.

%\begin{defn}[Visual Boundary]Let $X$ be a CAT(0) metric space.  For a fixed $x_0\in X$, let $\partial X$ be the set of geodesic rays with basepoint $x_0$.   For $\alpha\in\partial X$ and  $\epsilon,R>0$, let  $U(\alpha, R, \epsilon) = \bigl\{\beta\in \partial X \mid \d\bigl(\alpha(R), \im(\beta)\bigr)<\epsilon \bigr\}$.

%For $\alpha\in\partial X$ and  $\epsilon,R>0$, let  $\overline U(\alpha, R, \epsilon) = U(\alpha, R, \epsilon) \cup\bigl\{x \in X \mid [x_0, x]\cap \Ball_\epsilon\bigl(\alpha(R)\bigr) \neq \emptyset \bigr\}$.

%It is elementary to show that $\bigl\{U(\alpha,R, \epsilon)\bigr\}$ forms a basis for a topology on $\partial X$ which is called the \emph{visual topology}.  The \emph{visual boundary of $X$} is $\partial X$ endowed with the visual topology and will be denoted by $\partial_\infty X$.

%\end{defn}

\begin{defn}[Visual compactification and boundary]
Let $X$ be a CAT(0) metric space.   For a fixed $x_0\in X$, let $\partial X$ be the set of geodesic rays with basepoint $x_0$.

For $\alpha\in\partial X$ and  $\epsilon,R>0$, let  $$V(\alpha, R, \epsilon) = \bigl\{\beta\in \partial X \mid \d\bigl(\alpha(R), \im(\beta)\bigr)<\epsilon \bigr\} \cup\bigl\{x \in X \mid [x_0, x]\cap \Ball_\epsilon\bigl(\alpha(R)\bigr) \neq \emptyset \bigr\}.$$

If $X$ is a proper CAT(0) metric space (\emph{proper} meaning closed balls are compact), then $$\bigl\{ V(\alpha, R, \epsilon) \mid \alpha\in\partial X, \epsilon, R> 0\bigr\}\cup \bigl\{ \Ball_\epsilon(x) \mid x\in X, \epsilon> 0\bigr\}$$ is a basis for a compact topology on $\overline X = X\cup \partial X$, which we will call the \emph{visual compactification of $X$.}  Notice that $\partial X$ is a closed subspace.  The \emph{visual boundary of $X$} is the set $\partial X$ endowed with the subspace topology from $\overline X$ and will be denoted by $\partial_\infty X$.  The visual boundary has basis $\bigl\{U(\alpha, R, \epsilon) \mid \alpha\in\partial X, \epsilon, R> 0 \bigr\}$ where $U(\alpha, R, \epsilon) = V(\alpha, R, \epsilon)\cap \partial X$.
\end{defn}

\begin{defn}[Limit set]
  Let $X$ be a CAT(0) space and $A\subset X$.  The \emph{limit set} of $A$, denoted $\Lambda (A)$, is the closure of $A$ in the visual compactification intersected with the visual boundary, i.e. $\Lambda(A) = \cl_{\overline X} (A)\cap \partial_\infty X$ where $\cl_{\overline X}$ is the topological closure operator in $\overline X$.
\end{defn}

The visual boundary can be endowed with several natural metrics.  Here we will use the following metric since it relates well to the metric on asymptotic cones.   Fix $C>0$.  For two geodesics $\alpha, \beta: \bigl([0,\infty), 0\bigr) \to \left(X, x_0\right)$ into a CAT(0) space $X$, let $$\d_C(\alpha,\beta ) = \frac{1}{\sup\Bigl\{t \mid \d\bigl(\alpha(t),\beta(t)\bigl)\leq C\Bigr\}}.$$

\begin{lem}
The function $\d_C$ is a metric on $\partial X$.
\end{lem}

\begin{proof}
Notice that $\d_C(\alpha,\beta ) \leq \frac{1}{n}$ implies that $\d\bigl(\alpha(t),\beta(t)\bigl)\leq \frac{Ct}{n}$  for all $t\leq n$ by the CAT(0) condition.  Hence $\d_C(\alpha, \beta) = 0 $ if and only if $\alpha = \beta$.  Clearly $\d_C$ is symmetric and we are only left to show the triangle inequality holds for $\d_C$.

Suppose that $\d_C(\alpha,\beta) = \frac 1n$ and $\d_C(\beta, \gamma) = \frac 1m$ for some $m,n\in \mathbb R$.  For all $t\leq \min\{m,n\}$, we have $\d\left(\alpha(t),\gamma(t)\right) \leq Ct(\frac1n +\frac 1m)$ by the triangle inequality for the metric on $X$.  Hence $\d\left(\alpha(\frac{mn}{m+n}),\gamma(\frac{mn}{m+n})\right) \leq C$ which implies that
\[\d_C(\alpha, \gamma) \leq \frac{m+n}{mn} = \frac 1m +\frac 1n = \d_C(\alpha,\beta) +\d_C(\beta, \gamma) .\qedhere\]
\end{proof}

%We will use $B^{\d}_\epsilon (x)$ to denote the open metric ball of radius $\epsilon$ about $x$ with the metric $\d$.
When the metric is understood, we will simply write $B_\epsilon(\alpha)$ for $B^{\d_C}_\epsilon (\alpha)$.  Notice that $\Ball_{\epsilon/(CR)}^{\d_C}(\alpha)\subset U(\alpha, R, \epsilon)  $ as long as $0 < \epsilon \le C$, and $U (\alpha, 1/\epsilon + C/2, C/2) \subset \Ball_{\epsilon}^{\d_C}(\alpha)$, which proves the following observation.

\begin{lem}
The metric $\d_C$ induces the visual topology on $\partial X$.
\end{lem}

\begin{defn}[Tits/Angle boundaries] For points $x,y,z$ in a CAT(0) space $X$, we will let $\overline\angle _x(y,z)$ denote the comparison angle at $x$ between $y$ and $z$. If $p_y: [0,a] \to X$ and $p_z:[0,b] \to X$ are the unique geodesics in $X$ from $x$ to $y$ and from $x$ to $z$ respectively, then the \emph{angle between $y$ and $z$ at $x$} is $\angle _x(y,z) = \lim\limits_{t\to0} \overline\angle _x\bigl(p_y(t),p_z(t)\bigr)$.

For $\alpha, \beta\in\partial X$, let
\begin{align*}
\angle (\alpha, \beta) & = \lim\limits_{t\to \infty} \overline\angle_{x_0} \bigl(\alpha(t),\beta(t)\bigr) = \sup \bigl\{\overline\angle_{x_0} \bigl(\alpha(t),\beta(t')\bigr) \mid t,t'>0\bigr\} \text{ and } \\
\angle_{x_0} (\alpha, \beta)  & = \lim\limits_{t\to 0} \overline\angle_{x_0}\bigl(\alpha(t),\beta(t)\bigr) = \inf \bigl\{\overline\angle_{x_0}\bigl(\alpha(t),\beta(t')\bigr) \mid t,t'>0\bigr\}.
\end{align*}

It is an exercise to show that $\angle(\cdot,\cdot)$ defines a locally geodesic metric on $\partial X$, which is called the  \emph{angle} metric.  We will denote $\partial X$ with this metric by $\partial_\angle X$.

The path metric induced by $\angle(\cdot,\cdot)$ is called the Tits metric and is denoted by $\d_T(\cdot,\cdot)$.  The \emph{Tits boundary of $X$} is $\partial X$ with this metric and will be denoted by $\partial_T X$.  Note that $\d_T$ is an extended metric in the sense that it maps into $[0,\infty]$.  The Tits distance between any two points in distinct path components of the angle boundary is infinity. We refer the interested reader to \cite[Chapter I. 1, II.9]{BridHae99} for complete details.

\end{defn}

Proofs of the following standard lemmas can be found in \cite[Proposition II.9.8]{BridHae99} and \cite[Proposition II.9.9]{BridHae99}.

\begin{lem}Let $X$ be a CAT(0) space.  For $\alpha,\beta\in\partial X$,
$$ 2\sin\left (\dfrac{\angle(\alpha,\beta)}{2}\right) = \lim\limits_{t\to\infty} \dfrac{\d\bigl(\alpha(t),\beta(t)\bigr)}{t}.$$
\end{lem}

\begin{lem}[Flat Sector Theorem]
  Let $X$ be a CAT(0) space.  If $\alpha,\beta\in\partial X$ such that $\angle(\alpha,\beta)<\pi$ and $\angle(\alpha,\beta) = \angle_{x_0}(\alpha,\beta)$, then the convex hull of the geodesic rays
$\alpha$ and $\beta$ is isometric to a sector in the Euclidean plane bounded by two rays which meet at an angle $\angle(\alpha,\beta)$.
\end{lem}

\section{Metric Limits}

As we will only be concerned about limits of spaces indexed by a totally ordered sets, we will present the definitions for inverse and direct limits in terms of spaces indexed by a totally ordered set.

Let $(J, \ge)$ be a totally ordered set.  A \emph{system of metric spaces} is a collection of metric spaces $\{(X_j,\d_j)\mid j\in J\}$ together with a collection of non-expansive maps $\{f_{i,j}: X_i\to X_j \mid \text{ for all }i\geq j\}$, called the \emph{connecting maps}, with the property that $f_{i,k} = f_{i,j}\circ f_{j,k}$ for all $i\geq j\geq k$.

\begin{defn}[Direct limits of metric spaces]
Let $(X_j, \d_j)$ be a metric space for each $j \in J$.
Write $\bigsqcup\limits_{j\in J} X_j$ for the disjoint union of $X_j$.
Define a pseudometric on $\bigsqcup\limits_{j\in J} X_j$ by $\underline\d(x,y) = \inf \bigl\{\d_k \bigl(f_{i,k}(x), f_{j,k}(y)\bigr) \mid i,j \ge k\bigr\}$ for $x \in X_i$ and $y \in X_j$.
We will say that $x,y \in \bigsqcup\limits_{j\in J} X_j$ are equivalent, denoted $x \sim y$, if $\underline\d(x,y) = 0$.
The \emph{direct limit} of the system of metric spaces $\{X_i, f_{i,j}\}$ is the set $\bigsqcup\limits_{j\in J} X_j / \sim$ of equivalence classes, together with the metric $\underline\d$.
We will denote the direct limit of $\{X_i, f_{i,j}\}$ by $\lim\limits_\rightarrow \{X_i, f_{i,j}\}$ or when convenient $\lim\limits_\rightarrow X_i$.

It is immediate that the canonical induced maps $f_i:X_i\to \lim\limits_{\rightarrow} \{X_i, f_{i,j}\}$ are all non-expansive.
\end{defn}

\begin{defn}[Inverse limits of metric spaces]Recall that an extended metric space is a set with a distance function that potentially takes on infinite values but satisfies the rest of the properties of a metric.

If $(X_j, \d_j)$ is a metric space for each $j \in J$, then the metric product of $X_j$, written $\prod\limits_{j\in J} X_j$, is the extended metric space consisting of functions $x: J\to \bigcup\limits_{j\in J} X_j$ such that $x(j)\in X_j$ for all $j \in J$, with metric $\overline\d(x,y) = \sup_{j \in J} \d_j\bigl(x(j),y(j)\bigr)$.
The \emph{inverse limit} of a system of metric spaces is the metric subspace of $\prod\limits_{j\in J} X_j$ which consists of those functions $x$ for which $x(j) =f_{i,j}\bigl(x(i)\bigr)$ for every $i$ and $j$ such that $i \ge j$.  We will denote the inverse limit by $\lim\limits_{\leftarrow} \{X_i, f_{i,j}\}$ or when convenient by $\lim\limits_{\leftarrow} X_i$.

%It is immediate that the canonical induced maps $f_i:\lim\limits_{\leftarrow} \{X_i, f_{i,j}\} \to X_i$ are all non-expansive.
The inverse limit of complete metric spaces is complete; the direct limit need not be.
\end{defn}

\begin{rmk}
  We will show that the inverse limit of asymptotic cones is a metric space, with all points finite distance because the connecting maps all preserve distance to the basepoint.
\end{rmk}

\begin{rmk}
While the topological inverse/direct limit of a system of metric spaces is also a topological space, the metric on the metric inverse/direct limit need not always induce the topology of the topological inverse/direct limit.
\end{rmk}

\begin{exa}
Fix a basepoint $p$ in a complete CAT(0) space $X$.  The metric inverse limit of spheres $S_r(p)$ (with connecting maps coming from geodesic projection, and metric $\overline\angle_p$) is the Tits boundary %\footnote{Technically, the metric inverse limit is the sphere $S_1(0) \subset \cone(\partial_T X)$ of radius one about the cone point in the Euclidean cone over the Tits boundary, with the subspace metric.  One needs to either use the induced path metric on $S_1(0)$ to recover the Tits metric, or use the induced path metric on each sphere $\frac{1}{r} S_r(p)$ in the original system.}
$\partial_T X$.  On the other hand, the topological inverse limit of this system is the visual boundary $\partial_\infty X$.
Meanwhile, metric direct limit of this system of spheres is the space of directions\footnote{Indeed, $\Sigma_p X$ is formed by taking the set of germs of geodesic segments based at $p$, with pseudometric $\angle_p$.  Then $\Sigma_p X$ is the metric space of equivalence classes, where geodesic germs are considered equivalent if $\angle_p = 0$.}, $\Sigma_{p} X$, at $p$.  Its completion (if $X$ is proper and geodesically complete it is already complete) is called the link at $p$.

Now consider the system of closed metric balls $\overline\Ball_{r}(p)$ (with connecting maps all geodesic projection).  The metric inverse limit consists an isometric copy of $X$, along with the points of $\partial X$, all infinite distance from everything else in the space.  The topological inverse limit, on the other hand, is the visual compactification $\overline X = X \cup \partial_\infty X$ of $X$ \cite[p. 263]{BridHae99}.  Meanwhile, the metric direct limit is a single point.

%(With rescaled closed metric balls $\frac{1}{r}\overline\Ball_{1+r^2}(p)$ and suitable connecting maps, one can obtain all of $\cone(\partial_T X)$ as the metric inverse limit; $\cone(\Sigma_p X)$ as the metric direct limit; and $\overline X$ of $X$ as the topological inverse limit.)

If we use rescaled copies $\frac{1}{r} X$ of $X$ (using geodesic projection to rescale distances to $p$), the metric inverse limit is $\cone(\partial_T X)$, the Euclidean cone over the Tits boundary; the metric direct limit is the tangent cone $\cone(\Sigma_p X)$ at $p$, the Euclidean cone over the space of directions.  The system of asymptotic cones we will consider resembles this last system, but the details are more involved.
\end{exa}

\section{Direct limits and inverse limits of asymptotic cones}

\subsection*{Ultrafilters and Asymptotic cones}

\begin{defn}
A \emph{(non-principal) ultrafilter} $\omega$ on a set $S$ is a finitely additive probability measure on the power set of $S$ with values in $\{0,1\}$ such that $\omega(A) = 0$ for all finite subsets $A\subset S$. We will say that $A\subset S$ is \emph{$\omega$-large} if $\omega(A) = 1$.  A property $P$ holds \emph{$\omega$-almost surely} if it holds for some $\omega$-large subset of $S$.% and we will frequently write $\omega$-$P$ or $P_\omega$.

For $\{a_s\mid s\in S\}\subset X$, a subset of a topological space $X$ indexed by $S$, we will say that the \emph{ultralimit of $a_s$ is $x$}, written $\lim^\omega a_s = x$, if for every open neighborhood $U$ of $x$ the set $\{s \mid a_s \in U\}$ is $\omega$-large.  It is an exercise to show that every $S$-indexed subset of a compact space has a unique ultralimit and that ultralimits satisfy the standard properties of limits.
If $\{a_s\} \subset X$ is $S$-indexed but has no ultralimit in $X$, we will say $a_s$ is \emph{$\omega$-divergent}.
\end{defn}

\begin{defn}[Ultraproducts]
Let $(X_n, \d_n)$ be a sequence of metric spaces, $\omega$ an ultrafilter on $\mathbb N$, and $e=(e_n)\in \prod\limits_{n=1}^\infty X_n$.   The \emph{ultraproduct} of $(X_n, \d_n)$ is $$\prod {}^{{}^\omega}_{{}_e} X_n = \Bigl\{ (x_n)\in\prod_{n=1}^{\infty} X_n \ \Bigl | \ \text{for each $(x_n)$, $\d_n(x_n, e_n) $ is uniformly bounded}\Bigr\}/\sim$$  where $(x_n)\sim (y_n)$ if $\lim^\omega \d_n(x_n, y_n) = 0$.  The ultraproduct has metric
\[\d\bigl((x_n), (y_n)\bigr) = \lim\nolimits^\omega \d_n (x_n, y_n).\]

In general, the ultraproduct depends on both $e$ and $\omega$.  The sequence $e$ will be called the \emph{observation sequence} for the ultraproduct.  %When considering ultraproducts of uniformly bounded spaces or coarsely homogeneous spaces, $\prod^\omega_e X_n$ is independent of $e$.

We will use the simplified notation  $\partial_\angle^\omega X=  \prod^\omega_e \partial_\angle X$  and $\partial_T^\omega X =  \prod^\omega_e \partial_T X$.  Since $\partial_\angle X$ is bounded, $\partial_\angle ^\omega X$ is independent of the chosen basepoint but in the case of $\partial_T^\omega X$, there is an implied but unspecified choice of basepoint.

Let $\omega$ be an ultrafilter on $\mathbb N$ and $d=(d_n)$, an $\omega$-divergent sequence of positive real numbers (called a \emph{scaling sequence}). An \emph{asymptotic cone of $X$} is $\prod {}^{{}^\omega}_{{}_e} (X, \d/d_n)$ and will be denoted by  $\conn\bigl(X, e, d\bigr)$.
\end{defn}

\begin{defn}[Euclidean cones]\label{cones}  If $X$ is a metric space, let $\cone(X) =  (\mb R^+\times X) / \sim$ where $(0,x)\sim (0, x')$ for all $x,x'\in X$.  When convenient, we will denote the equivalence class of $(t,x)$ in $\cone(X)$ by $tx$.

We can endow $\cone(X)$ with a metric by $$\d^2(tx, t'x') = t^2 +(t')^2 -2tt'\cos\Bigl(\max\bigl\{\pi, \d(x,x')\bigr\}\Bigr).$$

Write  $\bigl(\cone(X)\bigr)^\omega$ for $\prod^\omega_e \cone(X)$ with $e= (0x)$.

\end{defn}

Notice that if $X$ is a complete CAT(0) space, then $\conn\bigl(X, e, d\bigr)$, $\cone(\partial_\angle X)$, $\cone(\partial_T X)$, $\cone(\partial^\omega_\angle X)$, $\cone(\partial^\omega_T X)$, $\bigl(\cone(\partial_\angle X)\bigr)^\omega$, and $\bigl(\cone(\partial_T X)\bigr)^\omega$ are all complete CAT(0) spaces.

\begin{prop}
Let $X$ be a complete CAT(0) space. The identity map from $\partial_\angle X $ to $\partial_TX$ induces an isometry from  $\cone(\partial_\angle X)$ to $\cone(\partial_T X)$. The natural map from $\cone(\partial^\omega_\angle X)$ to $\bigl(\cone (\partial_\angle X)\bigr)^\omega$ is an isometry.  As well, $\partial^\omega_T X$ is homeomorphic to a path component of $\partial^\omega_\angle X$.  \end{prop}

Hence we will often identify $\cone(\partial_\angle X)$ with $\cone(\partial_T X)$.
\begin{proof}The angle boundary, $\partial_\angle X$, is a CAT(1) space, see \cite[Theorem II.9.13]{BridHae99}.  Hence the identity map from $\partial_\angle X$ to $\partial_T X$ is a isometry when restricted to an open ball of radius $\pi$.  It is then immediate that $\cone(\partial_\angle X)$ is canonically isometric to $\cone(\partial_T X)$.

Let $F: \cone(\partial^\omega_\angle X)\to\bigl(\cone (\partial_\angle X)\bigr)^\omega$ by $F\bigl(t,(\alpha_n)\bigr) = \bigl((t,\alpha_n)\bigr)$.  It is a straightforward exercise to verify that $F$ is an isometry.

Recall that the ultraproduct $\partial^\omega_T X$ depends on a fixed observation sequence $e= (e_n)$ in $\partial_T X$ and is the set of sequences of geodesic rays $(\alpha_n)$ such that $\d_T(\alpha_n, e_n)$ is  uniformly bounded $\omega$-almost surely.  Thus for $(\alpha_n),(\beta_n)\in \partial^\omega_T X$, there is some $M$ such that $\d_T(\alpha_n, \beta_n)< M$ $\omega$-almost surely and the $\omega$-limit of the geodesics in $\partial_T X$ from $\alpha_n$ to $\beta_n$ gives a path from $(\alpha_n)$ to $(\beta_n)$ in $\partial^\omega_\angle X$.  Therefore the identity map from $\partial_T^\omega X$ to $\partial_\angle^\omega X$, takes $\partial_T^\omega X$ into a path component of $\partial_\angle^\omega X$. Suppose that $(\gamma_n)$ is any element of $\partial_\angle^\omega X$ contained in the same path component as $e$.  Since $\partial_\angle^\omega X$ is locally geodesic, there exists a rectifiable path from $(\gamma_n)$ to $e$.  Thus $\d_T(\gamma_n, e_n)$ is uniformly bounded $\omega$-almost surely  by the length of this rectifiable path, which implies that $(\gamma_n)\in \partial_T^\omega X$.  Therefore $\partial_T^\omega X$ is a path component of $\partial_\angle^\omega X$.
\end{proof}

\begin{cor}

For a complete CAT(0) space $X$  the following are equivalent.

\begin{enumerate}[(i)]
  \item The natural embedding of $\cone(\partial^\omega_T X)$ into $\cone(\partial^\omega_\angle X)$ is surjective.
  \item The diameter of $\partial_T X$ is finite.
\end{enumerate}

\end{cor}

\begin{defn}\label{embedding}%[$\Psi_d^\omega$]
Fix $x_0 \in X$.
For  a fixed scaling sequence $d = (d_n)$ and a non-principal ultrafilter $\omega$, we will define $\Psi^\omega_d: \cone(\partial^\omega_\angle X) \to \conn\bigl(X, (x_0), d\bigr)$ by $$\Psi^\omega_d \bigl(t, (\alpha_n)\bigr) = \bigl(\alpha_n(td_n)\bigr).$$
Similarly we define $\Psi_d: \cone(\partial_\angle X) \to \conn\bigl(X, (x_0), d\bigr)$ by $$\Psi_d \bigl(t, \alpha\bigr) = \bigl(\alpha(td_n)\bigr).$$
Notice that for a constant sequence of geodesic, we have $\Psi_d^\omega \bigl(t,(\alpha)\bigr) = \Psi_d(t,\alpha)$.  Thus when convenient, we will identify $\cone(\partial_\angle X)$ with its canonical diagonal embedding in $\cone(\partial^\omega_\angle X)$ and consider $\Psi_d$ as a restriction of $\Psi_d^\omega$.
\end{defn}

For fixed $(\alpha_n)\in \partial_\angle^\omega X$ and $\alpha\in \partial_\angle X$, the maps $t \mapsto \Psi^\omega_d\bigl(t,(\alpha_n)\bigr)$ and $t \mapsto \Psi_d(t,\alpha)$ are geodesics rays in $\conn\bigl(X, (x_0), d\bigr)$ based at $(x_0)$.
Thus we have the following induced maps of boundaries: \\
$\overline \Psi^\omega_d :\partial^\omega_\angle X \to \partial_\angle\Bigl(\conn\bigl(X, (x_0), d\bigr)\Bigr)$ by  $$\overline\Psi^\omega_d(\alpha_n)(t) = \Psi^\omega_d \bigl(t, (\alpha_n)\bigr)$$
and $\overline \Psi_d :\partial_T X \to \partial_T\Bigl(\conn\bigl(X, (x_0), d\bigr)\Bigr)$ by  $$\overline\Psi_d(\alpha)(t) = \Psi_d \bigl(t, \alpha\bigr).$$

Recall that a metric space $X$ is \emph{cocompact} if there exists a compact set $C$ such that the translates of $C$ by the full isometry group of $X$ cover $X$.

\begin{prop}\label{contracting}The maps $\Psi^\omega_d$ and $\overline\Psi^\omega_d$ are non-expansive.  If $X$ is proper and cocompact, then $\Psi^\omega_d$ and $\overline\Psi^\omega_d$ are surjective.
\end{prop}

\begin{proof}

By construction of the Euclidean cone, the natural map $\psi: \cone(\partial_\angle X) \to X$ given by $\psi(t,\alpha) = \alpha(t)$ is 1-Lipschitz.  It follows that $\Psi_d^\omega$ is 1-Lipschitz.

Let $(\alpha_n), (\beta_n)\in \partial_\angle^\omega X$.  For any $s>0$,
     \begin{align*}
        \d\bigl((\alpha_n), (\beta_n)\bigr) & = \lim_n{}^\omega \angle(\alpha_n, \beta_n) = \lim_n{}^\omega\lim\limits_{t\to\infty} \overline\angle_{x_0}\bigl(\alpha_n(t), \beta_n(t)\bigr)\\ &\geq \lim_n{}^\omega~\overline\angle_{x_0}\bigl(\alpha_n(sd_n), \beta_n(sd_n)\bigr) \\%\hspace{1in} \mbox{ for any } s>0\\
        &= \overline\angle_{(x_0)}\bigl(\overline\Psi_d^\omega(\alpha_n) (s),\overline\Psi_d^\omega(\beta_n) (s) \bigr).
     \end{align*}
Since this holds for any $s>0$,
$$ \d\bigl((\alpha_n), (\beta_n)\bigr)\geq \lim\limits_{s\to\infty} \overline\angle_{(x_0)}\bigl(\overline\Psi_d^\omega(\alpha_n) (s),\overline\Psi_d^\omega(\beta_n) (s) \bigr)= \d\bigl(\overline\Psi_d^\omega(\alpha_n), \overline\Psi_d^\omega(\beta_n)\bigr).$$
Thus $\overline\Psi_d^\omega$ is $1$-Lipschitz.

If $X$ is proper and admits a cocompact action, by \cite{GeogheganOntaneda2007} there is some constant $K$ (depending only on $X$ and the action) such that for every $(x_n)\in \conn\bigl(X, (x_0), d\bigr)$ there exists a sequence of geodesic rays $\alpha_n$ such that $\alpha_n(0) = x_0$ and $\d\bigl(\alpha_n(t_nd_n),  x_n\bigr)\leq K$ in $X$, where $t_n$ converges $\omega$-almost surely to $t = \d\bigl((x_0), (x_n)\bigr)$ in $\conn\bigl(X, (x_0), d\bigr)$.  Thus $\Psi_d^\omega\bigl(t, (\alpha_n)\bigr) = (x_n)$, i.e. $\Psi_d^\omega$ is surjective.

Showing that $\overline\Psi^\omega_d$ is surjective requires a bit more work.  Fix $\tilde \alpha$ a geodesic in $\conn\bigl(X, (x_0), d\bigr)$ based at $(x_0)$.  Since $\Psi^\omega_d$ is surjective, there exist $(\alpha_n^k)\in \partial_\angle^\omega X$ such that $\Psi^\omega_d\bigl(k,(\alpha_n^k)\bigr) = \tilde \alpha(k)$.  For $k\in \mathbb N$, fix a representative $(x_n^k)$ of $\tilde\alpha(k)$.  Let

$$A_i = \left\{ n \ \Bigl|\ \d(\alpha^j_n(jd_n), x^j_n \bigr)< \frac{d_n}{i} \text{ for all } j\leq i\right\}.$$

Notice that $A_i$ is $\omega$-large for each $i$ and forms a nested sequence.  Let $D_n = \bigl\{i\leq n \mid n\in A_i\bigr\}$ and $ m_n = \max D_n$ if $D_n\neq \emptyset$ and $1$ otherwise.  Then $m_n$ diverges $\omega$-almost surely and  $$\d\bigl(\alpha^{m_n}_n(m_nd_n), x_n^{m_n}\bigr) < \dfrac{d_n}{m_n}.$$

Since $m_n$ diverges $\omega$-almost surely, this together with the CAT(0) condition shows that for $\omega$-almost all $n$
$$\d\bigl(\alpha^{m_n}_n(kd_n), x_n^{k}\bigr) < \dfrac{d_n}{l},$$  for any fixed $k,l\in \mathbb N$.

Thus $\overline \Psi_d^\omega (\alpha^{m_n}_n) = \tilde \alpha$ and $\overline \Psi_d^\omega$ is surjective.
\end{proof}

\begin{lem}\label{flat sector}
For fixed $\alpha,\beta\in \partial X$,
$$\angle\bigl(\overline\Psi_d(\alpha), \overline\Psi_d(\beta)\bigr)= \angle\bigl(\alpha, \beta\bigr)= \angle_{(x_0)}\bigl(\overline\Psi_d(\alpha),\overline\Psi_d(\beta)\bigr).$$

In particular, if $\angle(\alpha, \beta)< \pi$, then the geodesics $\overline\Psi_d(\alpha)$ and $\overline\Psi_d(\beta)$ bound a Euclidean sector in $\conn\bigl(X, (x_0), d\bigr)$ for every scaling sequence $d$ and every ultrafilter $\omega$.
\end{lem}

\begin{proof}

By the Flat Sector Theorem, the second conclusion of the lemma will follow from the first.
The first comes from the pair of calculations
    \begin{align*}
        \angle\bigl(\overline\Psi_d(\alpha), \overline\Psi_d(\beta)\bigr)  &=  \lim\limits_{t\to\infty} \overline\angle_{(x_0)}\bigl(\Psi_d(t,\alpha), \Psi_d(t,\beta)\bigr) \\
        & =  \lim\limits_{t\to\infty} \lim_n{}^\omega ~\overline\angle_{x_0}\bigl(\alpha(td_n), \beta(td_n)\bigr) \\
        &= \lim\limits_{t\to\infty} \angle\bigl(\alpha, \beta\bigr) = \angle\bigl(\alpha, \beta\bigr)
     \end{align*}
and similarly
     \begin{align*}
        \angle_{(x_0)}\bigl(\overline\Psi_d(\alpha),\overline\Psi_d(\beta)\bigr)  &=  \lim\limits_{t\to 0} \overline\angle_{(x_0)}\bigl(\Psi_d(t,\alpha), \Psi_d(t,\beta)\bigr) \\
        & =  \lim\limits_{t\to0} \lim_n{}^\omega~ \overline\angle_{x_0}\bigl(\alpha(td_n), \beta(td_n)\bigr) \\
        & = \lim\limits_{t\to 0} \angle\bigl(\alpha, \beta\bigr) = \angle\bigl(\alpha, \beta\bigr).\qedhere
     \end{align*}
\end{proof}

\begin{prop}\label{isometric embedding}
The map $\Psi_d$ is an isometric embedding.  Thus $\Psi_d^\omega$ restricted to the diagonal is an isometric embedding.
\end{prop}

\begin{proof}
If $\angle (\alpha, \beta)<\pi$, the geodesics $\overline\Psi_d(\alpha) $ and $\overline\Psi_d(\beta)$ bound a Euclidean sector by Lemma \ref{flat sector} which implies that the metric on the Euclidean cone and the metric in the asymptotic cone agree.  If $\angle(\alpha,\beta) = \pi$ then $\d\bigl((s,\alpha), (t,\beta)\bigr) = s+t = \d\Bigl( \Psi_d(s, \alpha\bigr), \Psi_d(t, \beta) \Bigl)$.  Thus $\Psi_d$ is an isometric embedding.
\end{proof}

\begin{defn}
The ultrafilter $\omega$ induces a total order on the set of scaling sequences by $ (d_n') =d' \preceq d = (d_n)$ if $d_n'\leq d_n$ $\omega$-almost surely.
\end{defn}

\begin{lem}\label{countable}
Every countable set of scaling sequences is bounded with this ultrafilter ordering on the set of scaling sequences.
\end{lem}

\begin{proof}
  Let $\bigl\{d^i = (d_n^i)\bigr\}$ be a countable set of scaling sequences.  Let $\overline d = (\overline d_n) $ where $\overline d_n = \max\{ d_n^i\mid i\leq n\}$.  It is immediate that $d^i\preceq \overline d$ for all $i$.

  It takes more care to find a lower bound.  For each $i\in\mathbb N$, let $$A_i = \{n\in\mathbb N\mid d_n^j > i-1 \text{ for all } j\leq i\}.$$

  Since each scaling sequence is $\omega$-divergent, $A_i$ is the finite intersection of finitely many $\omega$-large set.  Hence, $A_i$ is $\omega$-large. The sets $A_i$ form a nested sequence, i.e. $\mathbb N = A_1\supset A_2\supset \cdots$, that we will use to defined $\underline d$.  Notice that $\bigcap_{i=1}^{\infty} A_i = \emptyset$, so $\mathbb N = \bigcup_{i=1}^{\infty} (A_i \backslash A_{i+1})$.

  For $n\in A_i\backslash A_{i+1}$, let $\underline d_n = \min\{ d_n^j \mid j\leq i\}$ and let $\underline d = (\underline d_n)$. Thus $\underline d_n \leq d^i_n$ on $A_i$, which implies that $\underline d \preceq d^i$ for each $i$.  The sequence $\underline d$ is $\omega$-divergent since $\underline d_n > i-1$ for all $n\in A_i$.
\end{proof}

\begin{lem}\label{isometric embedding2}
    For every countable subset $C\subset \partial_\angle^\omega X$, there exists a scaling sequence $d=(d_n)$ such that for all $(\tilde d_n)= \tilde d \succeq d$, $\Psi_{\tilde d}^\omega$ restricted to $\cone(C)$ is an isometric embedding.  In particular, there exists a scaling sequence $d$ such that if $(\alpha_n),(\beta_n)\in C$, then $\angle\left(\overline\Psi_{\tilde d}^\omega(\alpha_n), \overline\Psi_{\tilde d}^\omega(\beta_n)\right)  =\angle_{(x_0)}\left(\overline\Psi_{\tilde d}^\omega(\alpha_n), \overline\Psi_{\tilde d}^\omega(\beta_n)\right)$ for all $\tilde d \succeq d$.
\end{lem}

Notice that this implies that if $\lim\limits_{n}{}^\omega \angle (\alpha_n, \beta_n)< \pi$, then the geodesics $\Psi_{\tilde d}^\omega\bigl( t, (\alpha_n)\bigr)$ and $ \Psi_{\tilde d}^\omega( t, (\beta_n)\bigr)$ bound a Euclidean sector for all $\tilde d\succeq d$.

\begin{proof}
    Let $C$ be a countable subset of $\partial_\angle^\omega X$ and fix $(\alpha_n), (\beta_n)\in C$.  Then there exists $d_n$ such that $\overline\angle _{x_0}\bigl(\alpha_n(\sqrt{d_n}), \beta_n(\sqrt{d_n})\bigr)\geq \angle (\alpha_n, \beta_n)- 1/n$.  Thus for any $(\tilde d_n) = \tilde d \succeq d$ and $t>0$, we have $t \tilde{d}_n \ge t d_n \ge t \sqrt{d_n}$ $\omega$-almost surely, and therefore $\overline\angle _{x_0}\bigl(\alpha_n(t\tilde d_n), \beta_n(t\tilde d_n)\bigr)\geq \overline\angle _{x_0}\bigl(\alpha_n(\sqrt{d_n}), \beta_n(\sqrt{d_n})\bigr)\geq \angle (\alpha_n, \beta_n)- 1/n$ $\omega$-almost surely.  Then

    \begin{align*}
        \angle\left(\overline\Psi_{\tilde d}^\omega(\alpha_n), \overline\Psi_{\tilde d}^\omega(\beta_n)\right)  &=  \lim\limits_{t\to\infty} \overline\angle_{(x_0)}\bigl(\Psi_{\tilde d}^\omega(t,(\alpha_n)), \Psi_{\tilde d}^\omega(t,(\beta_n))\bigr) \\
        & =  \lim\limits_{t\to\infty} \lim_n{}^\omega \overline\angle_{x_0}\bigl(\alpha_n(t\tilde d_n), \beta_n(t\tilde d_n)\bigr) \\
        & \leq \lim\limits_{t\to\infty} \lim_n{}^\omega \angle\bigl(\alpha_n, \beta_n\bigr) =  \lim_n{}^\omega \angle\bigl(\alpha_n, \beta_n\bigr)
     \end{align*}
and similarly
     \begin{align*}
        \angle_{(x_0)}\left(\overline\Psi_{\tilde d}^\omega(\alpha_n), \overline\Psi_{\tilde d}^\omega(\beta_n)\right) &=  \lim\limits_{t\to 0} \overline\angle_{(x_0)}\bigl(\Psi_{\tilde d}^\omega\bigl(t,(\alpha_n)\bigr), \Psi_{\tilde d}^\omega\bigl(t,(\beta_n)\bigr)\bigr) \\
        & =  \lim\limits_{t\to0} \lim_n{}^\omega \overline\angle_{x_0}\bigl(\alpha_n(t\tilde d_n), \beta_n(t\tilde d_n)\bigr) \\
        & \geq \lim\limits_{t\to0} \lim_n{}^\omega \bigl( \angle\bigl(\alpha_n, \beta_n\bigr) -1/n \bigr) =  \lim_n{}^\omega \angle\bigl(\alpha_n, \beta_n\bigr).
     \end{align*}

    Thus $\angle\left(\overline\Psi_{\tilde d}^\omega(\alpha_n), \overline\Psi_{\tilde d}^\omega(\beta_n)\right)  =\angle_{(x_0)}\left(\overline\Psi_{\tilde d}^\omega(\alpha_n), \overline\Psi_{\tilde d}^\omega(\beta_n)\right)$.

    \textbf{Case 1: $\lim\limits_{n}{}^\omega \angle (\alpha_n, \beta_n)< \pi$}.  Then the geodesics $\Psi_{\tilde d}^\omega\bigl( t, (\alpha_n)\bigr)$ and $ \Psi_{\tilde d}^\omega( t, (\beta_n)\bigr)$ bound a Euclidean sector which implies that $\d\bigl(\Psi^\omega_d\bigl(t,(\alpha_n)\bigr),\Psi^\omega_d\bigl(s,(\beta_n)\bigr) = \d\Bigl(\bigl(t,(\alpha_n)\bigr), \bigl(s,(\beta_n)\bigr)\Bigr)$.

    \textbf{Case 2: $\lim\limits_{n}{}^\omega \angle (\alpha_n, \beta_n)= \pi$}.  As before, we have $\d\Bigl(\bigl(s,(\alpha_n)\bigr), \bigl(t,(\beta_n)\bigr)\Bigr) = s+t = \d\Bigl( \Psi_d^\omega(s, (\alpha_n)\bigr), \Psi_d^\omega(t, (\beta_n)) \Bigl)$.

    Thus for every pair of elements of $C$ there exists a scaling sequence satisfying the conclusion of the lemma.  Since the set of pairs from $C$ is also countable, Lemma \ref{countable} implies that there exists a sequence $d$ which will satisfy the conclusion of the lemma.
\end{proof}

%{\Large\blue{Kent: Start Here}}

\begin{defn}
    Fix $x_0 \in X$.  Define $\Xi: X\times\mathbb R^+ \to X$ by $$\Xi(x,t) = \begin{cases} x & \text{ for } t\geq \d(x_0, x)\\ y \in[x_0,x]\text{ such that } \d(x_0, y) = t & \text{ for } t\leq \d(x_0, x)\end{cases}.$$  Notice that $\Xi$ is the canonical geodesic retraction of $X$ to $x_0$ via projection onto the closed ball of radius $t$ about $x_0$.

    Fix scaling sequences $(d_n') =d'\preceq d= (d_n)$.   Then $\Xi$ defines a map $\Theta_{d'}^d : \conn\bigl(X, (x_0), d\bigr) \to \conn\bigl(X, (x_0), d'\bigr)$ by  $\Theta_{d'}^d \bigl((x_n)\bigl) = \Bigl(\Xi\bigl(x_n,\frac{d_n'}{d_n}\d(x_n,x_0)\bigr)\Bigr)$.
\end{defn}

\begin{lem}\label{commutes}
    Let $X$ be a CAT(0) metric space. Then $\Theta_{d'}^d$ is a well-defined 1-Lipschitz map which preserves distance to the observation point of the asymptotic cone.

    If, in addition, $X$ is proper and cocompact; then $\Theta _{d'}^d$ is surjective and $\Theta _{d'}^d\circ \Psi_d^\omega=  \Psi_{d'}^\omega$ for all $(d_n')= d'\preceq d= (d_n)$.
\end{lem}

\begin{proof}
    The maps $\Theta_{d'}^d$ are well-defined 1-Lipschitz maps, since $\Xi|_{X\times\{t\}}$ is 1-Lipschitz for any fixed $t$.   By construction, $\Theta_{d'}^d\bigl((x_n)\bigr)$ has a representative $(x_n')$ where $\d(x_n', x_0) = \frac{d_n'}{d_n}\d(x_n,x_0)$ $\omega$-almost surely.  Thus $$\d\bigl(\Theta_{d'}^d\bigl((x_n)\bigr), (x_0)\bigr) = \lim{}^\omega \dfrac{\d(x_n', x_0)}{d_n'} =\lim{}^\omega \dfrac{\d(x_n, x_0)}{d_n}= \d\bigl((x_n),(x_0)\bigr).$$

    Notice that $\Theta_{d'}^d \circ \Psi_d^\omega \bigl(t, (\alpha_n)\bigr) = \Theta_{d'}^d\bigl(\bigl(\alpha_n(td_n)\bigr)\bigr)=\bigl(\alpha_n(td_n')\bigr) = \Psi_{d'}^\omega \bigl(t, (\alpha_n)\bigr)$.  If $X$ is also proper and cocompact, then $\Psi^\omega_{d'}$ is surjective, which implies that is $\Theta_{d'}^d$ must also be surjective.
\end{proof}

Thus $\Theta^d_{d'}$, while not an isometry, does send geodesics in $\conn\bigl(X, (x_0), d\bigr)$ based at $(x_0)$ to  geodesics in $\conn\bigl(X, (x_0), d'\bigr)$ based at $(x_0)$.  When convenient, we can consider the function $\Theta^d_{d'}$ induces on the boundaries of  $ \conn\bigl(X, (x_0), d\bigr)$ and $ \conn\bigl(X, (x_0), d'\bigr)$.

\begin{cor}
  Let $X$ be a proper cocompact CAT(0) space.  Then the map  $$\overline\Theta_{d'}^d :\partial_\angle \conn\bigl(X, (x_0), d\bigr) \to \partial_\angle \conn\bigl(X, (x_0), d'\bigr)$$ defined by $\overline \Theta _{d'}^d \circ \overline\Psi_d^\omega( \alpha_n) = \overline\Psi_{d'}^\omega( \alpha_n)$  is a non-expansive surjection for all $d'\preceq d$.
\end{cor}

\begin{proof} Let $(d'_n) = d'\preceq d = (d_n)$.
  If $\overline\Psi_d^\omega( \alpha_n) = \overline\Psi_d^\omega( \beta_n)$, then $\bigl(\alpha_n(td_n)\bigr) = \bigl(\beta_n(td_n)\bigr)$ for all $t$ which implies that $\bigl(\alpha_n(td_n')\bigr) = \bigl(\beta_n(td_n')\bigr)$ for all $t$. Hence, $\overline\Psi_{d'}^\omega( \alpha_n) = \overline\Psi_{d'}^\omega( \beta_n)$.  Therefore $\overline\Theta_{d'}^d$ is well-defined.

   Then
   \begin{align*}
        \angle\bigl(\overline \Theta _{d'}^d \circ \overline\Psi_d^\omega( \alpha_n),\overline \Theta _{d'}^d \circ \overline\Psi_d^\omega( \beta_n)\bigr)
        &= \angle\bigl(\overline\Psi_{d'}^\omega( \alpha_n), \overline\Psi_{d'}^\omega( \beta_n)\bigr) \\
        %& = \lim\limits_{t\to\infty} \overline\angle_{(x_0)}\bigl(\overline\Psi_{d'}^\omega( \alpha_n)(t), \overline\Psi_{d'}^\omega( \beta_n)(t) \bigr)\\
        & = \lim\limits_{t\to\infty} \overline\angle_{(x_0)}\bigl(\Psi_{d'}^\omega\bigl(t, (\alpha_n)\bigr), \Psi_{d'}^\omega\bigl( t, (\beta_n)\bigr) \bigr)\\
        & = \lim\limits_{t\to\infty} \overline\angle_{(x_0)}\bigl(\Theta _{d'}^d\circ \Psi_d^\omega\bigl(t, (\alpha_n)\bigr), \Theta _{d'}^d\circ \Psi_d^\omega\bigl( t, (\beta_n)\bigr) \bigr)\\
        & \leq \lim\limits_{t\to\infty} \overline\angle_{(x_0)}\bigl(\Psi_d^\omega\bigl(t, (\alpha_n)\bigr), \Psi_d^\omega\bigl( t, (\beta_n)\bigr) \bigr)\\
        & = \angle\bigl(\overline\Psi_d^\omega( \alpha_n), \overline\Psi_d^\omega( \beta_n)\bigr)
   \end{align*}
   implies that $\overline\Theta_{d'}^d$ is a non-expansive map.  Since $\overline \Psi^\omega_{d'}$ is surjective, so is $\overline\Theta_{d'}^d$.
\end{proof}

Thus $\Bigl(\conn\bigl(X, (x_0), d\bigr), \Theta_{d'}^{d}\Bigr)$ forms a directed system of metric spaces and one can consider the direct limit $\lim\limits_{\rightarrow} \Bigl(\conn\bigl(X, (x_0), d\bigr), \Theta_{d'}^{d}\Bigr)$ and the inverse limit $\lim\limits_{\leftarrow} \Bigl(\conn\bigl(X, (x_0), d\bigr), \Theta_{d'}^{d}\Bigr)$.  Before computing the direct limit, we require the following lemma.

\begin{lem}\label{surjective}
    Let $X$ be a proper CAT(0) space.  For every $(\alpha_n) \in \partial_\angle^\omega X$, there exits a scaling sequence $d$ such that $\overline\Psi^{\omega}_{ d'} (\alpha_n) = \overline\Psi_{ d'}(\alpha)$ for all $ d'\preceq d$ where $\alpha = \lim^\omega \alpha_n$ (the limit is taken in $\partial_\infty X$).  Thus $\Psi^{\omega}_{ d'} \bigl(t,(\alpha_n)\bigr) = \Psi_{ d'}\bigl(t,\alpha\bigr)$ for all $t$ and all $ d'\preceq d$.
\end{lem}

\begin{proof}
  Let $\alpha = \lim^\omega \alpha_n$ where the limit is taken in the compact space $\partial_\infty X$.  Let $\tilde d_n = \max \{t\mid \d\bigl(\alpha(t), \alpha_n(t)\bigr)\leq C\}$, for some fixed $C>0$.  In other words $\tilde d_n = \frac{1} {\d_C(\alpha,\alpha_n)}$.  Since $\alpha_n$ converges to $\alpha$, $\omega$-almost surely, the sequence $d_n = \sqrt{\tilde{d}_n}$ diverges $\omega$-almost surely.  Let $ ( d_n') = d' \preceq d= (d_n)$.  Then $t \le d_n'$ $\omega$-almost surely, so
  \begin{align*}
    \d\bigl(\overline\Psi^{\omega}_{d'} (\alpha_n)(t) , \overline\Psi_{d'} (\alpha)(t)\bigr)
    &  = \d\bigl( (\alpha_n(t d_n'), (\alpha(t d_n')\bigr) \\
    & = \lim_n{}^\omega\dfrac{\d\bigl(\alpha_n(t d_n'), \alpha(t d_n')\bigr)}{ d_n'} \\
    &\leq  \lim_n{}^\omega\dfrac{\d\bigl(\alpha_n( d_n'^2), \alpha(d_n'^2)\bigr)}{ d_n'} \\%\hspace{1in} \mbox{ since } t\leq^\omega  d_n'\\
    &\leq \lim_n{}^\omega \frac{C}{d_n'} = 0.\qedhere
  \end{align*}
\end{proof}

\begin{cor}\label{stabilize}
  Let $X$ be a proper cocompact CAT(0) space.  Then for $x,y\in \conn\bigl(X, (x_0), d\bigr)$, there exists a scaling sequence $\tilde d\preceq d $ such that $ \d\bigl(\Theta^d_{\tilde d}(x), \Theta^d_{\tilde d}(y)\bigr) = \d\bigl(\Theta^d_{ d'}(x), \Theta^d_{ d'}(y)\bigr)$ for all $d'\preceq \tilde d$.  %In particular,  $ \d\bigl(\Theta^d_{\tilde d}(x), \Theta^d_{\tilde d}(y)\bigr) = \sup\limits_{ d}\inf\limits_{d'\preceq d} \bigl\{\d_i\bigl(f_i^{-1}\bigl(x\bigr),f_i^{-1}\bigl(y\bigr)\bigr)\bigr\}$

  %In addition, if $w,z\in \conn\bigl(X, (x_0), d''\bigr)$ have the property that for some $d'\preceq \tilde d$  $\Theta^d_{d' }(x)=\Theta^{d''}_{d'}(z)$, $\Theta^d_{d' }(y)=\Theta^{d''}_{d'}(w)$, then $ \d\bigl(\Theta^d_{d'}(x), \Theta^d_{d'}(y)\bigr) \leq \d\bigl(\Theta^{d''}_{ d'}(w), \Theta^{d''}_{ d'}(z)\bigr)$.
\end{cor}

\begin{proof}
   Fix $x,y\in \conn\bigl(X, (x_0), d\bigr)$.    Since $X$ is proper cocompact we can find $s,t\in \mathbb R^+$ and geodesics $\alpha_n,\beta_n$ such that $\Psi_d^\omega \bigl(s, (\alpha_n)\bigr) = x$ and $y =\Psi_d^\omega \bigl(t, (\beta_n)\bigr)$.  Let $\alpha = \lim^\omega \alpha_n$ and $\beta = \lim^\omega \beta_n$.  We can then find $\tilde d$ such that $\Theta_{\tilde d}^d (x) =\Psi^{\omega}_{\tilde d} \bigl(s,(\alpha_n)\bigr) = \Psi_{ \tilde d}\bigl(s,\lim^\omega\alpha_n\bigr)$ and $\Theta_{\tilde d}^d (y) =\Psi^{\omega}_{\tilde d} \bigl(t,(\beta_n)\bigr) = \Psi_{ \tilde d}\bigl(t,\lim^\omega\beta_n\bigr)$.  Thus for all $d'\preceq\tilde d $  \begin{align*}
     \d\bigl(\Theta^d_{\tilde d}(x), \Theta^d_{\tilde d}(y)\bigr) &= \d\bigl(\Psi_{ \tilde d}(s,\alpha), \Psi_{ \tilde d}(t,\beta)\bigr) \\ &= \d\bigl(\Psi_{ d'}(s,\alpha), \Psi_{ d'}(t,\beta)\bigr) = \d\bigl(\Theta^d_{ d'}(x), \Theta^d_{ d'}(y)\bigr).\qedhere
   \end{align*}
\end{proof}

 %   The direct limit, $\lim\limits_{\rightarrow} \Bigl(\conn\bigl(X, (x_0), d\bigr), \Theta_{d'}^{d}\Bigr)$, is the disjoint union of the asymptotic cones modulo the equivalence induced by connecting maps $\Theta_{d'}^d$, i.e.  $(x_n)\in \conn\bigl(X, (x_0), d\bigr)$ is equivalent to $(y_n) \in \conn\bigl(X, (x_0), d'\bigr)$ if there exists a scaling sequence  $d''$ with $d'' \preceq d$ and $d''\preceq d'$ such that $\Theta^d_{d''}(x_n) = \Theta^{d'}_{d''}(y_n)$.

\begin{thm}\label{direct limit}
    Let $X$ be a proper cocompact CAT(0) space.  Then $\lim\limits_{\rightarrow} \Bigl(\conn\bigl(X, (x_0), d\bigr), \Theta_{d'}^{d}\Bigr)$ is isometric to $\cone (\partial_TX)$.
Moreover, for each scaling sequence $d$ the projection map
\[\Theta_d: \conn\bigl(X, (x_0), d\bigr) \to \lim\limits_{\rightarrow} \Bigl(\conn\bigl(X, (x_0), d\bigr), \Theta_{d'}^{d}\Bigr)\]
is determined by the equation
$\Theta_d \circ \Psi^\omega_d (t, (\alpha_n)) = \Theta_d \circ \Psi_d (t, \lim\nolimits^\omega \alpha_n)$ and $\Theta_d\circ \Psi_d$ is an isometry.
\end{thm}

\begin{proof}
    Recall that by Lemma \ref{commutes}, for every $d'\preceq d$, we have that  $\Theta^d_{d'}\circ \Psi_d = \Psi_{d'}$.   Since $\Theta _{d'}^d\circ \Psi_d^\omega \bigl(t,(\alpha_n)\bigr) =  \Psi_{d'} (t, \lim\nolimits^\omega \alpha_n)$ for sufficiently small $d'$ (Lemma \ref{surjective}),  we see that  $\Theta_d \circ \Psi^\omega_d (t, (\alpha_n)) = \Theta_d \circ \Psi_d (t, \lim\nolimits^\omega \alpha_n)$.  Thus $\Theta_d\circ \Psi_d$ is surjective.

    By Proposition \ref{isometric embedding}, $\Psi_d$ is an isometric embedding of $\cone (\partial_TX)$ into $\conn\bigl(X, (x_0), d\bigr)$, which together with the equality $\Theta^d_{d'}\circ \Psi_d = \Psi_{d'}$ shows  $\Theta^d_{d'}$ is an isometry when restricted to the image  of $\Psi_{d}$. Thus $\Theta_d\circ\Psi_d$ is an isometry from $\cone (\partial_TX)$ to $\lim\limits_{\rightarrow} \Bigl(\conn\bigl(X, (x_0), d\bigr), \Theta_{d'}^{d}\Bigr)$.
\end{proof}

\begin{rmk}
%  We will identify $\cone (\partial_TX)$ with $\lim\limits_{\rightarrow} \Bigl(\conn\bigl(X, (x_0), d\bigr), \Theta_{d'}^{d}\Bigr)$ and  use $\Theta_d: \conn\bigl(X, (x_0), d\bigr) \to \cone(\partial_T X)$ to denote the projection map of $\conn\bigl(X, (x_0), d\bigr)$ onto $\cone (\partial_TX)$ induced by $\Theta_{d'}^d$.
  We can now identify $\cone (\partial_TX)$ with $\lim\limits_{\rightarrow} \Bigl(\conn\bigl(X, (x_0), d\bigr), \Theta_{d'}^{d}\Bigr)$.% and  use the map $\Theta_d: \conn\bigl(X, (x_0), d\bigr) \to \cone(\partial_T X)$ defined in Theorem \ref{direct limit}.
\end{rmk}

\begin{thm}\label{inverse limit}
    Let $X$ be a proper cocompact CAT(0) space.  Then $\lim\limits_{\leftarrow} \Bigl(\conn\bigl(X, (x_0), d\bigr), \Theta_{d'}^{d}\Bigr)$ is a complete metric space with the inverse limit metric and $\cone(\partial_\angle^\omega X)$ isometrically embeds into $\lim\limits_{\leftarrow} \Bigl(\conn\bigl(X, (x_0), d\bigr), \Theta_{d'}^{d}\Bigr)$.
\end{thm}

\begin{proof}
    Let $x,y \in \lim\limits_{\leftarrow} \Bigl(\conn\bigl(X, (x_0), d\bigr), \Theta_{d'}^{d}\Bigr)$.  This means for each scaling sequence $d$ we have $x^d ,y^d \in \conn\bigl(X, (x_0), d\bigr)$ such that $\Theta_{d'}^{d}(x^d)=x^{d'}$ and $\Theta_{d'}^{d}(y^d)=y^{d'}$ for all $d'\preceq d$.  The function assigning to each scaling sequence $d$ the real number $\d(x^d, y^d)$  is an increasing function of $d$ that is bounded by the constant $ \d\bigr(x^d, (x_0)\bigr)+ \d\bigr(y^d, (x_0)\bigr)$, which is independent of $d$ by Lemma \ref{commutes}.  Thus $\sup\limits_{d} \d(x^d, y^d)$ is finite and we can define a metric (not just an extended metric) on $\lim\limits_{\leftarrow} \Bigl(\conn\bigl(X, (x_0), d\bigr), \Theta_{d'}^{d}\Bigr)$ by $\rho\bigl((x^d), (y^d)\bigr) = \sup\limits_{d} \d(x^d, y^d)$.  It is easy to verify that $\rho$ does define a complete metric.  Lemma \ref{commutes} implies that the functions $\Psi_d^\omega$ induce a well-defined map into the inverse limit that is an isometric embedding by Lemma \ref{isometric embedding2}.
\end{proof}

\section{Visual boundary}\label{visual section}

We have seen that $\Psi_d^\omega : \partial_\angle^\omega X \to \conn\bigl(X, (x_0), d\bigr)$ gives a parametrization of $\conn\bigl(X, (x_0), d\bigr)$ which converges to $\cone(\partial_T X)$ as we allow $d$ to decrease.  Thus $\bigl\{ \conn\bigl(X, (x_0), d\bigr), \Theta_{d'}^d\bigr\}_{d,d'}$ completely determines the Tits boundary $\partial_T(X)$.  We now wish to understand the visual boundary $\partial_\infty X$ in terms of $\bigl\{ \conn\bigl(X, (x_0), d\bigr), \Theta_{d'}^d\bigr\}_{d,d'}$.

\begin{lem}\label{converse}
Let $X$ be a proper CAT(0) space and $(\alpha_n)$ a sequence in $\partial X$.  If $\Psi_d^\omega \bigl(t_0,(\alpha_{n})\bigr) = \Psi_d(t_0,\alpha_0)$ for some scaling sequence $d$, some $t_0>0$ and some $\alpha_0 \in \partial X$, then $\lim^\omega \alpha_n  = \alpha_0$.
\end{lem}

\begin{proof}
Suppose that $\Psi_d^\omega \bigl(t_0,(\alpha_{n})\bigr) = \Psi_d(t_0,\alpha_0)$ for some $t_0>0$.  Then $$\lim\limits_n{}^\omega \dfrac{\d\bigl(\alpha_n(t_0d_n), \alpha_0(t_0d_n)\bigr)}{d_n} = 0.$$  Fix $C, \epsilon>0$.  Let $A$ be the $\omega$-large set such that $\d\bigl(\alpha_n(t_0d_n), \alpha_0(t_0d_n)\bigr) \le C \epsilon t_0 d_n$ for all $n \in A$.  Then the CAT(0) inequality implies that $\d\bigl(\alpha_n(\frac{1}{\epsilon}), \alpha_0(\frac{1}{\epsilon})\bigr) \le C$, and therefore $\d_C(\alpha_n, \alpha_0) \le \epsilon$, for all $n\in A\cap B$ where $B$ is the $\omega$-large set such that $d_n \ge \frac{1}{\epsilon t_0}$ for all $n\in B$.  Thus $\lim^\omega\alpha_n = \alpha_0$.
\end{proof}

\begin{thm}\label{visual recovery}
    Let $X$ be a proper CAT(0) space.  For a sequence $(\alpha_n)$ in $\partial X$, $\alpha_n$ converges to $\alpha_0$ in the visual boundary of $X$ if and only if for every bijection $\sigma: \mathbb N \to \mathbb N$, there exists a scaling sequence $d$ such that $\Psi_d^\omega \bigl(t,(\alpha_{\sigma(n)})\bigr) = \Psi_d(t,\alpha_0)$ for all $t$.
\end{thm}

Notice that $\Psi_d^\omega \bigl(t,(\alpha_{\sigma(n)})\bigr) = \Psi_d(t,\alpha_0)$ for all $t$ is equivalent to $\Theta_d\circ \Psi_d^\omega \bigl(t,(\alpha_{\sigma(n)})\bigr) = (t,\alpha_0)$ for all $t$.

\begin{proof}
    If $\alpha_n$ converges to $\alpha_0$ in the visual boundary, then $\lim^\omega \alpha_{\sigma(n)} = \alpha_0$ for all bijections $\sigma$ and the forward implication then follows from Lemma \ref{surjective}.  Thus we need only show that if for every bijection $\sigma: \mathbb N \to \mathbb N$, there exists a scaling sequence $d$ such that $\Psi_d^\omega \bigl(t, (\alpha_{\sigma(n)})\bigr) = \Psi_d(t,\alpha_0)$ for all $t$, then $\alpha_n$ converges to $\alpha_0$ in the visual boundary.

Suppose that there exists a subsequence $n_i$ such that $\alpha_{n_i}$ converges to $\beta$.  We may assume that $B= \{n_i\}$ has infinite complement in $\mathbb N$.  Let $A$ be an $\omega$-large subset of $\mathbb N$ with infinite complement.  Let $\sigma: \mathbb N \to \mathbb N$ be a bijection sending $B$ to $A$.

Then $\lim^\omega \alpha_{\sigma(n)} = \beta$.  By hypothesis, there also exists a scaling sequence $d$ such that $\Psi_d^\omega \bigl(t,(\alpha_{\sigma(n)})\bigr) = \Psi_d(t,\alpha_0)$ for all $t$ which by Lemma \ref{converse} implies that $\lim^\omega \alpha_{\sigma(n)} = \alpha_0$.  Thus $\alpha_0 = \beta$.

It is then an exercise to use the compactness of $\partial_\infty X$ to show that $\alpha_n$ converges to $\alpha_0$, if every convergent subsequence of $\alpha_n$ converges to $\alpha_0$.
\end{proof}

\section{Quasi-isometries}\label{quasi-isometries}

A map $f: X\to Y$ of metric spaces is a \emph{$(\lambda,C)$-quasi-isometric embedding} if
$$\frac{1}{\lambda}\d(x,y) - C\leq \d\bigl(f(x),f(y)\bigr)\leq \lambda\d(x,y) + C \qquad \textnormal{ for every } x,y\in X$$
and is \emph{$C$-quasi-surjective} if $Y \subset \mathcal N_C\bigl(\im(f)\bigr)$, i.e., the open $C$-neighborhood of $\im(f)$ is all of $Y$.
A \emph{$(\lambda,C)$-quasi-isometry} is a $C$-quasi-surjective $(\lambda,C)$-quasi-isometric embedding.  A \emph{$C$-quasi-inverse} of a map $f: X\to Y$ is a map $g:Y\to X$ such that $f\circ g $ is $C$-close to the identity on $Y$ and $g\circ f$ is $C$-close to the identity on $X$.
It is a standard exercise to show that every quasi-isometry $f: X\to Y$ admits a
quasi-inverse $g: Y\to X$, which is itself a quasi-isometry.  A \emph{quasi-geodesic ray} in $Y$ is a quasi-isometric embedding of $\mathbb R^+$ into $Y$.

\begin{defn}
    Let $\omega$  be a nonprincipal ultrafilter on $\mathbb N$ and $d = (d_n)$ be an $\omega$-divergent sequence of positive real numbers.  Suppose that $f:X\to Y$ is a $(\lambda, C)$-quasi-isometric embedding.   Then $f$ naturally induces a $\lambda$-Lipschitz map $f^\omega :\conn\bigl(X, (x_0), d\bigr) \to \conn\bigl(Y, \bigl(f(x_0)\bigr), d\bigr)$, defined by $f^\omega\bigl((x_n) \bigr) = \bigl(f(x_n)\bigr)$.  If $f$ is a quasi-isometry, $f^\omega$ is bi-Lipschitz.

    On the other hand, a quasi-isometric embedding (or even a quasi-isometry) $f:X\to Y$ of proper CAT(0) spaces does not induce a canonical map from $\partial X$ to $\partial Y$.  The reason is that if $\alpha:\mathbb R^+\to X $ is a geodesic ray, then $f\circ\alpha$ is a quasi-geodesic ray, which in general will not have a unique limit point in $\partial_\infty Y$ (see Proposition \ref{non-stabilize}).  However, since $Y$ is proper, the pointwise limit $f^d(\alpha) = \lim^\omega[f\circ\alpha(0), f\circ\alpha(d_n)]$ is a geodesic ray in $Y$.  This defines a map $f^d: \partial X \to \partial Y$.  In general, $ f^d$ will depend on both $\omega$ and the scaling sequence $d$.
\end{defn}

\begin{rmk}
  The functions $\Theta_d, \Theta^d_{d'}, \Psi_d, \Psi^\omega_d$ are defined for all CAT(0) spaces and have domains and ranges which depend on the chosen CAT(0) space.  In most cases we will allow the chosen CAT(0) space to change without changing our notation for the functions $\Theta_d, \Theta^d_{d'}, \Psi_d, \Psi^\omega_d$.   \emph{When considering a quasi-isometric embedding  $f: X\to Y$, we will always assume that the fixed basepoint in $Y$ used to define the maps $\Theta_d, \Theta^d_{d'}, \Psi_d, \Psi^\omega_d$ is the image of the fixed basepoint in $X$.}
\end{rmk}

We will now relate the functions $\Psi_d, \Theta_d,$ and $ f^d$.

%\begin{lem}
%    Let $f:X\to Y$ be a $(\lambda,C)$-quasi-isometric embedding of a proper CAT(0) spaces, $\omega$ a nonprincipal ultrafilter on $\mathbb N$, and $d = (d_n)$ an $\omega$-divergent sequence of positive real numbers.  For every $\alpha\in \partial X$ there exists a scaling sequence $\tilde d = (\tilde d_n)$ such that for all $d' \preceq \tilde d$ and all $t \in \mathbb R^+$,
%\stepcounter{thm}
%    \begin{equation}\Theta^d_{d'}\circ f^\omega\circ\Psi_d(t,\alpha) = \Psi_{d'}\bigl(t, f^{d}(\alpha)\bigr).\end{equation}
%    In particular,
%\stepcounter{thm}
%    \begin{equation}\Theta_{d}\circ f^\omega\circ\Psi_d(t,\alpha) = \bigl(t, f^{d}(\alpha)\bigr).\end{equation}
%\end{lem}
%
%
%
%
%\begin{proof}
%    Let $f:X\to Y$ be a $(\lambda,C)$-quasi-isometric embedding of proper CAT(0) spaces, $\omega$ a nonprincipal ultrafilter on $\mathbb N$, and $d = (d_n)$ an $\omega$-divergent sequence of positive real numbers.  Fix $\alpha\in \partial X$ and let $\gamma_n$ be the geodesic from $ f\circ\alpha(0)$ to $f\circ\alpha(d_n)$.  Let $\tilde d_n=\max \{t\mid \d\bigl( f^d(\alpha)(t^2), \gamma_n(t^2)\bigr)\leq 1\}$.  Notice that $\tilde d_n$ is $\omega$-divergent since $\gamma_n$ converges to $ f^d(\alpha)$ by definition. For any $ (d_n')= d' \preceq \tilde d= (\tilde d_n)$, we have
%
%     \begin{align*}
%       \Theta^d_{d'}\circ f^\omega\circ\Psi_d(t,\alpha) & = \Theta^d_{d'}\bigl(f\circ\alpha(td_n)\bigr) \\
%        & = \bigl(\gamma_n(td_n')\bigr) = \bigl( f^d(\alpha)(td_n')\bigr) = \Psi_{d'}\bigl(t, f^{d}(\alpha)\bigr).
%     \end{align*}
%\end{proof}

\begin{lem} \label{lem - relations}
    Let $f:X\to Y$ be a quasi-isometric embedding of a proper CAT(0) spaces, $\omega$ a nonprincipal ultrafilter on $\mathbb N$, and $d = (d_n)$ an $\omega$-divergent sequence of positive real numbers.  For every $\alpha\in \partial X$ there exists a scaling sequence $\tilde d = (\tilde d_n)$ such that for all $d' \preceq \tilde d$,
\stepcounter{thm}
    \begin{equation}\Theta^d_{d'}\circ f^\omega\circ\Psi_d(1,\alpha) = \Psi_{d'}\bigl(1, f^{d}(\alpha)\bigr).\end{equation}
    In particular,
\stepcounter{thm}
    \begin{equation}\Theta_{d}\circ f^\omega\circ\Psi_d(1,\alpha) = \bigl(1, f^{d}(\alpha)\bigr).\end{equation}
\end{lem}

\begin{proof}
    Let $f:X\to Y$ be a quasi-isometric embedding of proper CAT(0) spaces, $\omega$ a nonprincipal ultrafilter on $\mathbb N$, and $d = (d_n)$ an $\omega$-divergent sequence of positive real numbers.  Fix $\alpha\in \partial X$ and let $\gamma_n$ be the geodesic from $f\circ\alpha(0)$ to $f\circ\alpha(d_n)$.  Let $\tilde d_n=\max \{t\mid \d\bigl( f^d(\alpha)(t^2), \gamma_n(t^2)\bigr)\leq 1\}$.  Notice that $\tilde d_n$ is $\omega$-divergent since $\gamma_n$ converges to $ f^d(\alpha)$ by definition.

For $d'= (d_n') \preceq (\tilde d_n)$,  we have
     \begin{align*}
       \Theta^d_{d'}\circ f^\omega\circ\Psi_d(1,\alpha)
       & = \Theta^d_{d'}\bigl(f\circ\alpha(d_n)\bigr)
          = \Theta^d_{d'}\bigl(\gamma_n(d_n)) \\
       & = \bigl(\gamma_n(d_n')\bigr)
       = \bigl( f^d(\alpha)(d_n')\bigr)
       = \Psi_{d'}\bigl(1, f^{d}(\alpha)\bigr).\qedhere
     \end{align*}
\end{proof}

\begin{prop}
    Let $f:X\to Y$ be a $(L,C)$-quasi-isometric embedding of a proper CAT(0) spaces, $\omega$ a nonprincipal ultrafilter on $\mathbb N$, and $d = (d_n)$ an $\omega$-divergent sequence of positive real numbers.
    The function $ f^d$ is $L$-Lipschitz continuous when the boundaries are endowed with the Tits or angle metric.
\end{prop}

\begin{proof}
%  Let $\alpha,\beta \in \partial X$.  Then
%  \begin{align*}
%    2\sin\left (\dfrac{\angle(\alpha,\beta)}{2}\right)  &= \lim\limits_{t\to\infty} \dfrac{\d\bigl(\alpha(t),\beta(t)\bigr)}{t} = \lim\limits_{n\to\infty} \dfrac{\d\bigl(\alpha(d_n),\beta(d_n)\bigr)}{d_n}  \\
%  & =  \d\bigl(\Psi_d(1, \alpha), \Psi_d(1,\beta)\bigr)\\
%  & \geq \frac{1}{L} \d\bigl(f^\omega\circ\Psi_d(1, \alpha), f^\omega\circ\Psi_d(1,\beta)\bigr)\\
%  & \geq \frac{1}{L} \d\bigl(\Theta_{d'}^d\circ f^\omega\circ\Psi_d(1, \alpha), \Theta_{d'}^d\circ f^\omega\circ\Psi_d(1,\beta)\bigr)\\
%  & = \frac{1}{L} \d\bigl(\Psi_{d'}\bigl(1, f^{d}(\alpha)\bigr), \Psi_{d'}\bigl(1, f^{d}(\beta)\bigr)\bigr)\\
%  & = \frac{1}{L}\lim\limits_{n\to\infty}  \dfrac{\d\bigl(f^d(\alpha)(d'_n),f^d(\beta)(d'_n)\bigr)}{d'_n}  \\
%%  & = \frac{1}{L}\lim\limits_{t\to\infty}  \dfrac{\d\bigl(f^d(\alpha)(t),f^d(\beta)(t)\bigr)}{t}  \\
%  &= \frac{2}{L}\sin\left (\dfrac{\angle\bigl(f^d(\alpha),f^d(\beta)\bigr)}{2}\right).
%    \end{align*}
First observe that for all $\alpha, \beta \in \partial X$ and any $\omega$-divergent sequence $d'= (d_n')$,
\begin{align*}
    2\sin\left (\dfrac{\angle(\alpha,\beta)}{2}\right)
 &= \lim\limits_{t\to\infty} \dfrac{\d\bigl(\alpha(t),\beta(t)\bigr)}{t} \\
 &= \lim\limits_{n}{}^{\omega} \dfrac{\d\bigl(\alpha(d_n'),\beta(d_n')\bigr)}{d_n'}
  =  \d\bigl(\Psi_{d'}(1, \alpha), \Psi_{d'}(1,\beta)\bigr).
\end{align*}
Fix $\alpha, \beta \in \partial X$ and any $\omega$-divergent sequence $\tilde d$  satisfying Lemma \ref{lem - relations}.  We have
  \begin{align*}
  2\sin\left (\dfrac{\angle\bigl(f^d(\alpha),f^d(\beta)\bigr)}{2}\right)
  &= \d\bigl(\Psi_{\tilde d}\bigl(1, f^{d}(\alpha)\bigr), \Psi_{\tilde d}\bigl(1, f^{d}(\beta)\bigr)\bigr) \\
  & = \d\bigl(\Theta^d_{d'}\circ f^\omega\circ\Psi_d(1, \alpha), \Theta^d_{d'}\circ f^\omega\circ\Psi_d(1,\beta)\bigr)\\
  & \leq L \d\bigl(\Psi_d(1, \alpha), \Psi_d(1,\beta)\bigr) \\
  & = 2L\sin\left (\dfrac{\angle(\alpha,\beta)}{2}\right).
    \end{align*}
Thus for every $\epsilon > 0$, we find $\angle\bigl(f^d(\alpha),f^d(\beta)\bigr) \le L (1 + \epsilon) \angle(\alpha,\beta)$ for all sufficiently close $\alpha, \beta \in \partial_\angle X$.  Hence $f^d$ is $L (1 + \epsilon)$-Lipschitz by the triangle inequality.  It follows that $f^d$ is $L$-Lipschitz on $\partial_\angle X$.  Since the angle and Tits metrics are equal wherever either has value $< \pi$, $f^d$ is $L$-Lipschitz on $\partial_T X$, too.
\end{proof}

In general $ f^d$ will not define a continuous function on the visual boundary, as illustrated by the following example.

\begin{exa}\label{example}
    Let $G$ be the group introduced by Croke and Kleiner in \cite{CrokeKleiner00} that acts properly cocompactly on CAT(0) spaces $X$ and $Y$, which have non-homeomorphic visual boundaries.  Then there exist $G$-equivariant quasi-isometries $f: X\to Y $ and $g: Y\to X$ such that both $f\circ g$ and $g\circ f$ are uniformly close to the identity map.  Since $G$ has rank one, the set of Morse geodesics of both $\partial_\infty X$ and $\partial_\infty Y$ are dense.

    If for some $d$, both $f^d$ and $g^d$ induced continuous maps on the visual boundary, then Lemma \ref{induced_inverse} would imply that $f^d\circ g^d$ is the identity on a dense set of $\partial_\infty Y$ and hence the identity function on all of $\partial_\infty Y$.  Similarly, $g^d\circ f^d$ would be the identity function on $\partial_\infty X$.  Thus $f^d$ would give a homeomorphism from $\partial_\infty X$ to $\partial_\infty Y$, a contradiction.  Thus, for every $d$, at least one of $f^d:\partial_\infty X\to \partial_\infty Y$ and $g^d:\partial_\infty Y\to \partial_\infty X$ is not continuous.

    Alternatively, Ruane and Bowers give an example of two actions of a group $G$ on the product of a tree with $\mathbb R$ such that the $G$-equivariant induced quasi-isometry does not induce a continuous map of the boundary \cite{BowersRuane96}.
\end{exa}

In \cite{Staley12}, it is shown that the limit set of a geodesic under a quasi-isometry can be any connected, compact subset of Euclidean space.  One might hope that for the sufficiently small scaling sequences the maps $f^d$ would tend to choose a favorite limit point, i.e. $f^d$ stabilize for a fixed $\alpha$ and sufficiently small scaling sequence $d$.  However the following proposition illustrates that this is not the case.

\begin{prop}\label{non-stabilize}
  Let $f:X\to Y$ be a quasi-isometric embedding of proper CAT(0) spaces and fix $\alpha\in \partial X$.  Then $\Lambda\bigl(\im(f\circ\alpha)\bigr) = \{ f^d(\alpha)\mid d\text{ is a scaling sequence} \}$. For every $\gamma \in \Lambda\bigl(\im(f\circ\alpha)\bigr)$ and scaling sequence $d$, there exist scaling sequences $d'$, $\tilde d$ such that $d'\preceq d\preceq \tilde d$ and   $ f^{\tilde d}(\alpha)= \gamma=  f^{d'}(\alpha) $.
\end{prop}

\begin{proof}
    Let $f:X\to Y$ be a quasi-isometric embedding of CAT(0) spaces  and fix a scaling sequence $d= (d_n)$.  Let $\gamma\in \Lambda\bigl(\im(f\circ\alpha)\bigr)$, i.e. $\gamma$ is in the limit set of $f\circ \alpha$.  Then there exist $\bar d_n$ such that $f\circ\alpha(\bar d_n)$ converges (not just $\omega$-almost surely) to $\gamma$ in $\overline Y$ which implies that $\lim^\omega \bigl[f\circ\alpha(0), f\circ\alpha(\bar d_n)\bigr] = \gamma$.

    Let $d_n' = \max \{ \bar d_i \mid \bar d_i \leq d_n\}$  and $\tilde d_n = \min \{ \bar d_i \mid \bar d_i \geq d_n\}$.  Then for $d' = (d_n')$ and $\tilde d= (\tilde d_n)$, we have $d'\preceq d\preceq \tilde d$ and it is elementary to check that $f^{d'}(\alpha) = \gamma=f^{\tilde d}(\alpha)$.
\end{proof}

%\begin{proof}
  %Let $f:X\to Y$ be a quasi-isometry of proper cocompact CAT(0) spaces and fix a scaling sequence $d$.  Suppose that $g: \mathbb S^n \to \partial_T X$ such that $f^d  \circ g$ extends to a map $h: \mathbb B^n \to \partial_TY$.
%\end{proof}

%In the rare case that the Tits boundary is compact we can say more.

\begin{thm}\label{t:compact 1}
Let $X$ be a proper cocompact CAT(0) space such that $\partial_T(X)$ is compact.  Then $\partial_T(X)$ is homeomorphic to $\partial_\infty X$ and if there exists a quasi-isometry $f: X\to Y$, then $f^d$ is a bi-Lipschitz equivalence from $\partial_T X$ to $\partial_T Y$.
\end{thm}

\begin{proof}
There is always a continuous bijection from $\partial_T X$ to $\partial_\infty X$.  If $\partial_T X$ is compact, then this is a homeomorphism.

Suppose there exists a quasi-isometry $f: X\to Y$.  Fix a scaling sequence $d = (d_n)$. If $\partial_T X$ is compact, then $\partial_\angle^\omega X$ is isometric to $\partial_T X$ and  $\cone (\partial_\angle^\omega X)$ is locally compact.  Thus both  $\conn\bigl(X, (x_0), d\bigr)$ and $\conn\bigl(Y, \bigl(f(x_0)\bigr), d\bigr)$ are proper metric spaces, since $\Psi^\omega_d$ is a 1-Lipschitz surjection that preserves distances to the basepoint.  Then $\Psi_d (1, \cdot)$ embeds $\partial_T Y$ as a closed subset of a closed ball of the proper metric space, $\conn\bigl(Y, \bigl(f(x_0)\bigr), d\bigr)$.  Thus  $\partial_T Y$ is also compact and $\partial_\angle^\omega Y = \partial_T Y$.

Thus $\Psi_d$ is a surjection when viewed as a map from $\cone (\partial_T X)$ to $\conn\bigl(X, (x_0), d\bigr)$ or as a map from $\cone (\partial_T Y)$ to $\conn\bigl(Y, \bigl(f(x_0)\bigr), d\bigr)$. By Lemma \ref{isometric embedding}, $\Theta_d$ is an isometry when restricted to the image of $\Psi_d$.  Thus   $f^d =  \Theta_d\circ f^\omega \circ \Psi_d |_{\{1\}\times \partial_T X}$ is bi-Lipschitz.
\end{proof}

It follows that compactness of the Tits boundary is a quasi-isometry invariant.  However, this restriction is so severe it essentially only holds for Euclidean flats.
More precisely, we give the following alternate proof of an unpublished result of Bosch\'{e} \cite[Proposition 3]{Bosche_preprint}.

\begin{thm}
Let $X$ be a proper CAT(0) space admitting a properly discontinuous, cocompact group action by isometries.  If $\partial_T X$ is compact then $X$ has a quasi-dense Euclidean flat.
\end{thm}

\begin{proof}
As noted in the proof of Theorem \ref{t:compact 1}, this implies that $\conn\bigl(X, (x_0), d\bigr)$ is a proper metric space, which implies that $\conn\bigl(X, (x_0), d\bigr)$ is a Euclidean flat by \cite{Point, Sapir15}.  Thus the Tits boundary of $X$ is a sphere by Lemma \ref{contracting} and Proposition \ref{isometric embedding}.  Thus by \cite[Proposition 2.1]{Leeb00}, $X$ contains a flat $F$ with $\partial_\infty F = \partial_\infty X$.

We now only need show that $F$ is quasi-dense.  By \cite{GeogheganOntaneda2007} there is some constant $K$ (depending only on $X$ and the action) such that for every $x, y\in X$ there exists a geodesic ray $\alpha$ based at $x$ such that  $\d\bigl(\im(\alpha),  y\bigr)\leq K$.  Suppose that there exists a point $y\in X$ such that $\d(y, F) \geq 2K$.  Let $x\in F$ be the closest point projection of $y$ in $F$ and $\alpha$ a geodesic ray based at $x$ such that  $\d\bigl(\im(\alpha),  y\bigr)\leq K$.  Then $\angle_x \bigl(y, \alpha\bigr) \le \pi/6$ by comparison with Euclidean geometry (law of sines), and $\angle_x(y, z) \geq \pi/2 $ for any $z\in F\backslash \{x\}$.  Thus for any geodesic ray $\beta$ in $F$ based at $x$, we have that $\angle_x\bigl(\alpha,\beta\bigr)\geq \pi/4$.  Thus $\alpha\not\in \partial F$, which contradicts the fact that $\partial_\infty F = \partial_\infty X$.  Thus $X\subset \mathcal N_{2K}(F)$ and $F$ is quasi-dense.
\end{proof}

Thus by Bieberbach's Theorem \cite{Bieberbach11}, we have the following corollary.

\begin{cor}
Every finitely generated group which acts properly discontinuously, cocompactly, and by isometries on a CAT(0) space with compact Tits boundary is virtually free abelian.
\end{cor}

The \emph{Hausdorff distance} between two subsets $A,B$ of a metric space $X$ is
$$\d_H\bigl(A, B\bigr) = \max\Bigl\{\sup\limits_{a\in A}\bigl\{\inf \limits_{b\in B} \d(a,b)\bigr\}, \sup\limits_{b\in B}\bigl\{\inf \limits_{a\in A} \d(a,b)\bigr\}\Bigr\}. $$

\begin{lem}\label{close rays}
    Let $\alpha$ and $\beta$ be quasi-geodesic rays in an arbitrary metric space.
If $\im (\alpha) \subset \mathcal N_{M} \bigl(\im (\beta) \bigl)$ for some $M$, then the Hausdorff distance $\d_{H} \bigl(\im (\alpha), \im (\beta)\bigr)$ between $\im (\alpha)$ and $\im (\beta)$ is finite.
\end{lem}

\begin{proof}
    Let $X$ be a metric space, and let $\alpha$ and $\beta$ be $(L',C')$-quasi-geodesic and $(L,C)$-quasi-geodesic rays, respectively, in $X$ such that $\im (\alpha) \subset \mathcal N_{M} \bigl(\im (\beta) \bigl)$ for some $M$.

    We need to show there exists $M'$ such that $\im (\beta) \subset \mathcal N_{M'} \bigl(\im (\alpha) \bigl)$.
Let $t_0= 0$ and, for $n\in \mathbb N$, fix $t_n$ such that $\d\bigl(\beta(t_n), \alpha(n)\bigr)\leq M$.
Notice that $t_n$ is unbounded  which implies that the union of the closed intervals between $t_n$ and $t_{n+1}$ is $\mathbb R^{\geq0}$.
By the triangle inequality, $\d\bigl(\beta(t_n), \beta(t_{n+1})\bigr) \leq 2M +(L'+C')$ which implies that $|t_n- t_{n+1}|\leq L(2M+L'+C'+C)$.

    Thus for every $t$ in the interval between $t_n$ and  $t_{n+1}$, we have $\d\bigl(\beta(t), \beta(t_n)\bigr)\leq L^2(2M +L'+C' +C) +C$ and hence $\d\bigl(\beta(t), \alpha(n)\bigr) \leq L^2(2M +L'+C' +C) +C+M$.

       Therefore $\im(\beta) \subset \mathcal N_{L^2(2M +L'+C' +C) +C+M} \bigl(\im(\alpha)\bigr)$.  This proves the lemma.
\end{proof}

\begin{defn}[Morse]
    A (quasi-)geodesic $\gamma$ is called a \emph{Morse (quasi-)geodesic}, if  for every $L\geq 1$ and $C\geq 0$ there exists an $M= M(L, C)$ such that every $(L,C)$-quasi-geodesic with endpoints on $\gamma$ remains $M$-close to $\gamma$.
\end{defn}

%\begin{defn} (contracting quasi-geodesics) A (quasi-)geodesic $\gamma$ is said to be \emph{(b,c)-contracting} if there exists constants $0<b\leq 1$ and $0<c$ such that $\forall x,y \in X,$ $$\d(x,y)<bd(x,\pi_{\gamma}(x))  \implies \d(\pi_{\gamma}(x),\pi_{\gamma}(y))<c.$$ For the special case of a $(b,c)$-contracting quasi-geodesic where $b$ can be chosen to be $1$, the quasi-geodesic $\gamma$ is called \emph{strongly contracting.} \end{defn}

\begin{lem}\label{morse1}
    Let $f: X \to Y$ be a $(L, C)$-quasi-isometry of proper CAT(0) spaces.  If $\alpha$ is a Morse ray in $\partial X$, then $f^d(\alpha)$ is a Morse ray in $\partial Y$ and there exists an $M$ such that $\d_{H} \bigl(f\circ \alpha, f^d(\alpha)\bigr) <M$.
\end{lem}

\begin{proof}
    It is immediate that $f\circ \alpha$ is Morse.
So for every $n$, $[f\circ\alpha(0), f\circ\alpha(d_n)]\subset \mathcal N_{M'}\bigl(\im(f\circ\alpha)\bigr)$  for some $M'$ depending only on $\alpha, L, C$, which implies that $\im\bigl(f^d(\alpha)\bigr) \subset \mathcal N_{M'+1}\bigl(\im(f\circ\alpha)\bigr)$.
By Lemma \ref{close rays}, $\im(f\circ \alpha)$ and $\im\bigl(f^d(\alpha)\bigr)$ are finite Hausdorff distance apart.
Since any quasi-geodesic ray that is finite Hausdorff distance from a Morse quasi-geodesic ray is also Morse, $f^d(\alpha)$ is a Morse geodesic ray at $f(x_0)$.
\end{proof}

\begin{lem}\label{induced_inverse}
  Let $f: X \to Y$ be a quasi-isometry of proper CAT(0) metric spaces with quasi-inverse $g: Y\to X$.  If $\alpha \in \partial X$ is a Morse ray, then $f^d(\alpha) = f^{d'}(\alpha) $ and $g^{d'}\circ f^d(\alpha) = \alpha$ for all scaling sequences $d,d'$.
\end{lem}

\begin{proof}
  Since $\d_{H}\bigl(f\circ \alpha, f^d(\alpha)\bigr),\d_{H}\bigl(f\circ \alpha, f^{d'}(\alpha)\bigr)$ are both finite then $f^d(\alpha)$ and $f^{d'}(\alpha)$ are asymptotic geodesics based at $f(x_0)$.  Hence $f^d(\alpha) = f^{d'}(\alpha) $.

  Since $g$ is a quasi-isometry and $f^d(\alpha)$ is a Morse geodesic, $\d_{H}\bigl(g\circ f^d(\alpha), g^{d'}\circ f^d(\alpha)\bigr)$ and $\d_{H}\bigl(g\circ f(\alpha), g\circ f^d(\alpha)\bigr)$ are both finite.
Since $g$ is a quasi-inverse of $f$, $\d_{H}\bigl(\alpha, g\circ f(\alpha))$ is also finite.
Thus $g^{d'}\circ f^{d}(\alpha)$ and $\alpha$ are asymptotic, and therefore equal.
\end{proof}

Although Morse rays are well-behaved under quasi-isometries, other rays are generally not.

\begin{exa}\label{spiral}
Consider the family of maps $f \colon X \to X$ on $X = \mathbb{R}^2$ given in polar coordinates by $f(r, \theta) = (r, \theta + h(r))$, where $h$ is differentiable on $r > 0$.
(The case $h(r) = \log(r)$ is the classic ``logarithmic spiral'' map.)
A short calculation reveals the Jacobian of $f$ at $(x,y) = (r \cos \theta, r \sin \theta)$ is, up to multiplying on the left and right by orthogonal matrices,
\[
\begin{pmatrix}
1 & 0 \\
0 & r
\end{pmatrix}
\begin{pmatrix}
1 & 0 \\
h'(r) & 1
\end{pmatrix}
\begin{pmatrix}
1 & 0 \\
0 & \frac{1}{r}
\end{pmatrix}
\\
=
\begin{pmatrix}
1 & 0 \\
r h'(r) & 1
\end{pmatrix}
\]
for $(x,y) \neq (0,0)$.
Thus we see that $f$ is Lipschitz if and only if $r h'(r)$ is bounded.
If so, $f^{-1}$ is also clearly well-defined and Lipschitz, so $f$ is bi-Lipschitz and therefore a quasi-isometry.

Again, the case $h(r) = \log(r)$ is the classic ``logarithmic spiral'' map, and it is easy to see that $f^d$ can take any given point in $\partial X$ to any other point of $\partial X$ by adjusting $d$.

The case $h(r) = \log(\log(r))$ also has the property that $f^d$ can map any point in $\partial X$ to any other by choosing the appropriate $d$.
However, in this case, for every $d$ the induced bi-Lipschitz map $f^\omega :\conn\bigl(X, (x_0), d\bigr) \to \conn\bigl(X, \bigl(f(x_0)\bigr), d\bigr)$ is an isometry, since $r h'(r) \to 0$ as $r \to \infty$.
\end{exa}

More examples can be constructed from the above by using the following general lemma.

\begin{lem}
Let $Y$ be a geodesic metric space and $\pi : Y \to X$ be a locally isometric covering map.
If $\widetilde f : Y \to Y$ is a continuous lift of an $L$-Lipschitz map $f : X \to X$, then $\widetilde f$ is $L$-Lipschitz.
\end{lem}
\begin{proof}
Assume $L > 0$ (the case $L = 0$ is trivial).
By the triangle inequality, it suffices to prove $\widetilde f$ is locally $L$-Lipschitz---that is, for every $y \in Y$ there is some $\delta > 0$ such that $\widetilde f \vert_{\Ball_\delta(y)}$ is $L$-Lipschitz.
So let $y \in Y$ and $x = \pi(y)$.
By hypothesis on $\pi$, there exists $\epsilon > 0$ such that
$\pi \vert_{\Ball_\epsilon(\widetilde f(y))} :
\Ball_\epsilon(\widetilde f(y)) \to \Ball_\epsilon(f(x))$
is an isometry, and $\delta > 0$ such that
$\pi \vert_{\Ball_\delta(y)} :
\Ball_\delta(y) \to \Ball_\delta(x)$
is an isometry.
By continuity of $\widetilde f$, we may assume $\delta > 0$ is small enough that $\widetilde f (\Ball_\delta(y)) \subset \Ball_\epsilon(\widetilde f(y))$.

So let $p,q \in \Ball_\delta(y)$.
Since $\pi \vert_{\Ball_\epsilon(\widetilde f(y))}$ and
$\pi \vert_{\Ball_\delta(y)}$ are isometries, we find
\begin{align*}
\d(\widetilde f(p), \widetilde f(q))
&= \d(\pi(\widetilde f(p)), \pi(\widetilde f(q)))
= \d(f(\pi(p)), f(\pi(q))) \\
&\le L \cdot \d(\pi(p), \pi(q))
= L \cdot \d(p, q).
\end{align*}
Thus we have shown that $\widetilde f$ is locally $L$-Lipschitz, which proves the lemma.
\end{proof}

\begin{cor}\label{lifting}
Let $X$ be a locally CAT(0) metric space with universal cover $\widetilde X$.
If $f : X \to X$ is bi-Lipschitz then its lift $\widetilde f : \widetilde X \to \widetilde X$ is bi-Lipschitz.
\end{cor}

\begin{exa}\label{seaweed}
Let $X = \mathbb{R}^2 \setminus \Ball_1(x)$ and let $\widetilde X$ be its universal cover with the induced CAT(0) metric, which is proper.
Consider the family of maps $f \colon X \to X$ on $X$ given in polar coordinates by $f(r, \theta) = (r, \theta + h(r))$, where $h$ is differentiable on $r \ge 1$.
From Example \ref{spiral}, we see that $f$ is bi-Lipschitz if and only if $r h'(r)$ is bounded, so by Corollary \ref{lifting}, $\widetilde f : \widetilde X \to \widetilde X$ is bi-Lipschitz under the same conditions.

Now notice that $\partial \widetilde X$ is isometric (under the Tits metric) to the disjoint union of a line and two points, corresponding to the directions $\theta \in \mathbb{R}$ and $\theta = \pm \infty$, respectively.
The isolated points $\theta = \pm \infty$ are Morse, so they are fixed by every $f^d$.
The other points, however, are not.

The case $h(r) = \log(r)$ has $f^d (\theta) = +\infty$ for all $\theta \in \mathbb{R} \subset \partial \widetilde X$.
So does $h(r) = \log(\log(r))$, and in this case $\widetilde f^\omega$ is always an isometry.
Thus while Morse rays map to Morse rays, they may not be the only rays that do.

The case $h(r) = \log(r) \sin(\log(\log(r)))$ has, for any given $\theta_1 \in \mathbb{R}$ and $\theta_2 \in \mathbb{R} \cup \{\pm \infty\}$, some $d$ such that $f^d (\theta_1) = \theta_2$.
So does $h(r) =\log(\log(r)) \sin(\log(\log(\log(r))))$, and in this case $\widetilde f^\omega$ is always an isometry.
Thus even the Tits component of $f^d(\alpha)$ is, in general, not solely determined by $f$ and $\alpha$, but can also depend on $d$.
\end{exa}

\begin{prop}\label{obvious}
Let $\alpha,\beta\in\partial X$ for $X$ a proper CAT(0) space.  Then the following are equivalent.
\begin{enumerate}[(i)]
	\item\label{1} The Tits distance from $\alpha$ to $\beta$ is finite. %$\d_T(\alpha, \beta)$ is finite.
	
	\item\label{2} For every (or some) $d$, $\d_T\bigl(\overline\Psi_d(\alpha), \overline\Psi_d(\beta)\bigr)$ is finite.

	\item\label{3}For every (or some) $d$ and any $t>0$, $\Psi_d(t,\alpha), \Psi_d(t,\beta)$ are contained in the same component of $\conn\bigl(X, (x_0), d\bigr)\backslash \bigl\{(x_0)\bigr\}$.
\end{enumerate}

\end{prop}

\begin{proof}
	%We will show that (\ref{3}) implies (\ref{2}) implies (\ref{1}) which implies (\ref{3}).
	
	The implication  $(\ref{1})\Rightarrow(\ref{2})$ follows immediately from Lemma \ref{flat sector}, and $(\ref{2})\Rightarrow(\ref{3})$ is trivial.  Thus we need only show $(\ref{3})\Rightarrow(\ref{1})$.

Let $\gamma': [0,1]\to \conn\bigl(X, (x_0), d\bigr)\backslash \bigl\{(x_0)\bigr\}$ be a path from $\Psi_d(1,\alpha)$ to $\Psi_d(1,\beta)$.  Since $\d\bigl(\im(\gamma'), (x_0)\bigr)>0$, we can use the geodesic retraction of $\conn\bigl(X, (x_0), d\bigr)\backslash \bigl\{(x_0)\bigr\}$ to find a path $\gamma: [0,1]\to \conn\bigl(X, (x_0), d\bigr)\backslash \bigl\{(x_0)\bigr\}$ from $\Psi_d(t_0,\alpha)$ to $\Psi_d(t_0,\beta)$ such $\d\bigl(\gamma(s), (x_0)\bigr) = t_0$ for all $s$.

Then we can fix $k$ sufficiently large such that $\d\bigl(\gamma(\frac ik), \gamma(\frac{i+1}{k})\bigr)< t_0$ for all $0 \leq i\leq k-1$.  Fix $\alpha^i_n $ a geodesic from $x_0$ to $x_n^i$ such that $(x_n^i) = \gamma(\frac ik)$.  Let $\alpha_i = \lim_n^\omega \alpha^i_n$.  Notice that $\alpha = \alpha_0$ and $\beta = \alpha_k$.  Fix $l\in \mathbb N$.  Then there exists an $\omega$-large set, $A$, such that $\d\bigl(\alpha_i(l), \alpha_n^i(l)\bigr)< 1$  and $\d\bigl(\alpha_n^i(t_0d_n), \alpha_n^{i+1}(t_0d_n)\bigr)\leq \frac{3}{2}t_0d_n$ for all $n\in A$ and for $0\leq i\leq k$.

Thus for any $n\in A$ such that $t_0d_n>l$, we have $$\dfrac{\frac32 t_0d_n + 2}{ t_0d_n}\geq \dfrac{\d\bigl(\alpha_i(l), \alpha_{i+1}(l)\bigr)}{l}$$ by the CAT(0) inequality.  Thus $\angle(\alpha_i, \alpha_{i+1})< \pi$ and $\d_T(\alpha, \beta)< k\pi$.
\end{proof}

\begin{cor}
   Let $X$ be a proper CAT(0) space.  Then $\conn\bigl(X, (x_0), d\bigr)$ has a cut-point if and only if the Tits diameter of $X$ is infinite.
\end{cor}

Since, under a cocompact action, having infinite Tits diameter is equivalent to having a periodic Morse geodesic, we have the following.

\begin{cor}Let $X$ be a proper cocompact CAT(0) space.
If $\conn\bigl(X, (x_0), d\bigr)$ has a cut-point, then all asymptotic cones of $X$ have cut-points and $X$ contains a periodic Morse geodesic.
\end{cor}

\section{Questions}
Theorem \ref{inverse limit} shows that the limit of asymptotic cones contains a canonical copy of $\cone(\partial_\angle^\omega X)= \bigl(\cone(\partial_T X)\bigr)^\omega$ for proper cocompact CAT(0) spaces.  However, it is not clear  even in simple cases, for example where $X$ is a tree, if this embedding is surjective.

\begin{quest}
    Is the inverse limit $\lim\limits_{\leftarrow} \Bigl(\conn\bigl(X, (x_0), d\bigr), \Theta_d^{d}\Bigr)$ isometric to $\cone(\partial_\angle^\omega X)$?
\end{quest}

Example \ref{example} shows that in general $f^d$ cannot be a homeomorphism.  However, it would be interesting to consider if $f^d$ is a homotopy equivalence or if $f^d$ preserves homotopy groups.

\begin{quest}
  Let $f:X\to Y$ be a quasi-isometry of proper cocompact CAT(0) spaces.  Does the continuous map $f^d : \partial_TX \to \partial_TY$ induces an injective homomorphism of homotopy groups?
\end{quest}

%\blue{We should be able to answer some of the questions asked in \cite{Sultan14}.}

\bibliographystyle{plain}
\bibliography{../../../bib}

\def\cprime{$'$} \def\cprime{$'$}
\begin{thebibliography}{10}

\bibitem{Bieberbach11}
Ludwig Bieberbach.
\newblock \"uber die {B}ewegungsgruppen der {E}uklidischen {R}\"aume.
\newblock {\em Math. Ann.}, 70(3):297--336, 1911.

\bibitem{Bosche_preprint}
Aur\'{e}lien Bosch\'{e}.
\newblock Tits compact cat(0) spaces.
\newblock {\em http://arxiv.org/abs/1106.0149v2}.

\bibitem{BowersRuane96}
Philip~L. Bowers and Kim Ruane.
\newblock Boundaries of nonpositively curved groups of the form {$G\times{\bf
  Z}^n$}.
\newblock {\em Glasgow Math. J.}, 38(2):177--189, 1996.

\bibitem{BridHae99}
Martin~R. Bridson and Andr{\'e} Haefliger.
\newblock {\em Metric spaces of non-positive curvature}, volume 319 of {\em
  Grundlehren der Mathematischen Wissenschaften [Fundamental Principles of
  Mathematical Sciences]}.
\newblock Springer-Verlag, Berlin, 1999.

\bibitem{CrokeKleiner00}
Christopher~B. Croke and Bruce Kleiner.
\newblock Spaces with nonpositive curvature and their ideal boundaries.
\newblock {\em Topology}, 39(3):549--556, 2000.

\bibitem{GeogheganOntaneda2007}
Ross Geoghegan and Pedro Ontaneda.
\newblock Boundaries of cocompact proper {${\rm CAT}(0)$} spaces.
\newblock {\em Topology}, 46(2):129--137, 2007.

\bibitem{Gromov1}
M.~Gromov.
\newblock Groups of polynomial growth and expanding maps.
\newblock {\em Inst. Hautes \'Etudes Sci. Publ. Math.}, (53):53--73, 1981.

\bibitem{gr2}
M.~Gromov.
\newblock Asymptotic invariants of infinite groups.
\newblock In {\em Geometric group theory, {V}ol.\ 2 ({S}ussex, 1991)}, volume
  182 of {\em London Math. Soc. Lecture Note Ser.}, pages 1--295. Cambridge
  Univ. Press, Cambridge, 1993.

\bibitem{Leeb00}
Bernhard Leeb.
\newblock {\em A characterization of irreducible symmetric spaces and
  {E}uclidean buildings of higher rank by their asymptotic geometry}, volume
  326 of {\em Bonner Mathematische Schriften [Bonn Mathematical Publications]}.
\newblock Universit\"at Bonn, Mathematisches Institut, Bonn, 2000.

\bibitem{Pansu}
P.~Pansu.
\newblock Croissance des boules et des g\'eod\'esiques ferm\'ees dans les
  nilvari\'et\'es.
\newblock {\em Ergodic Theory Dynam. Systems}, 3(3):415--445, 1983.

\bibitem{Point}
F.~Point.
\newblock Groups of polynomial growth and their associated metric spaces.
\newblock {\em J. Algebra}, 175(1):105--121, 1995.

\bibitem{Sapir15}
Mark Sapir.
\newblock On groups with locally compact asymptotic cones.
\newblock {\em Internat. J. Algebra Comput.}, 25(1-2):37--40, 2015.

\bibitem{Staley12}
Dan Staley.
\newblock Erratic behavior of {${\rm CAT}(0)$} geodesics under
  {$G$}-equivariant quasi-isometries.
\newblock {\em Geom. Dedicata}, 159:169--184, 2012.

\end{thebibliography}

\end{document}